\documentclass[12pt,reqno]{amsart}
\usepackage{amsthm,amsfonts,amssymb,amscd,amsmath,euscript}
\usepackage[all]{xy}
\usepackage{array}
\usepackage{amsmath}
\usepackage{amsfonts}
\usepackage{amssymb}
\usepackage{stmaryrd}
\usepackage{enumerate}
\usepackage{amsthm}
\usepackage{amsmath, amscd}
\usepackage{xy}
\xyoption{all}
\usepackage{supertabular}
\usepackage{marvosym}
\usepackage{wasysym}
\usepackage{hyperref}
\usepackage{MnSymbol}
\usepackage{tikz}
\usetikzlibrary{matrix,arrows,backgrounds}

\newtheorem{theorem}{Theorem}[section]

\newtheorem{lemma}[theorem]{Lemma}
\newtheorem{proposition}[theorem]{Proposition}
\newtheorem{corollary}[theorem]{Corollary}

\newtheorem{conjecture}[theorem]{Conjecture}

\theoremstyle{definition}
\newtheorem{definition}[theorem]{Definition}
\newtheorem{definition-lemma}[theorem]{Definition/Lemma}
\newtheorem{remark}[theorem]{Remark}

\newtheorem{example}[theorem]{Example}

\newcommand{\op}[1]{\operatorname{#1}}

\newcommand{\leftexp}[2]{{\vphantom{#2}}^{#1}{#2}}

\newcommand*{\DashedArrow}[1][]{\mathbin{\tikz [baseline=-0.25ex,-latex, dashed,#1] \draw [#1] (0pt,0.5ex) -- (1.3em,0.5ex);}}%

\newcommand{\dbcoh}[1]{\operatorname{D}^{\operatorname{b}}(\operatorname{coh }#1)}

\newcommand{\sidenote}[1]{}

\def\Z{\mathbb{Z}}

\def\Q{\mathop{\mathbb{Q}}}

\def\O{\mathcal{O}}

\def\P{\mathbb{P}}

\def\tand{\text{ and } }

\def\PP{{\mathbb P}}
\def\CC{{\mathbb C}}
\def\PP{{\mathbb P}}
\def\QQ{{\mathbb Q}}
\def\ZZ{{\mathbb Z}}

\def\cO{\mathcal{O}}
\def\cJ{\mathcal{J}}
\def\cO{\mathcal{O}}

\def\cX{\mathcal{X}}

\def\cG{\mathcal{G}}
\def\cY{\mathcal{Y}}
\def\cF{\mathcal{F}}
\def\cU{\mathcal{U}}
\def\cH{\mathcal{H}}

\begin{document}

\title{On the Griffiths groups of Fano manifolds of Calabi-Yau Hodge type}

\author[Favero]{David Favero}
\address{
  \begin{tabular}{l}
   David Favero \\
      \hspace{.1in} Universit\"at Wien, Fakult\"at f\"ur Mathematik\\
   \hspace{.1in}   Garnisongasse 3/14 O3.49a, Wien, \"Osterreich \\
   \hspace{.1in} Email: {\bf favero@gmail.com} \\
  \end{tabular}
}

\author[Iliev]{Atanas Iliev}
\address{
  \begin{tabular}{l}
   Atanas Iliev  \\ 
   \hspace{.1in} Department of Mathematics, Seoul National University \\
   \hspace{.1in} Seoul 151-747, Korea \\
   \hspace{.1in} Email: {\bf ailiev@snu.ac.kr} \\
  \end{tabular}
}

\author[Katzarkov]{Ludmil Katzarkov}
\address{
  \begin{tabular}{l}
   Ludmil Katzarkov \\
   \hspace{.1in} University of Miami, Department of Mathematics\\ 
   \hspace{.1in} PO Box 249085, Coral Gables, FL 33124-4250 USA \\
      \hspace{.1in} Universit\"at Wien, Fakult\"at f\"ur Mathematik \\
   \hspace{.1in}   Garnisongasse 3/14 O3.52, Wien, \"Osterreich \\
   \hspace{.1in} Email: {\bf lkatzark@math.uci.edu} \\
  \end{tabular}
}

\maketitle

\begin{center}
Dedicated to the memory of our friend A. Todorov
\vspace{12pt}
\end{center}

\begin{abstract}
A deep result of Voisin asserts that the Griffiths group of a general non-rigid Calabi-Yau (CY) 3-fold is infinitely generated.  This theorem builds on an earlier method of hers which was implemented by Albano and Collino to prove the same result for a general cubic sevenfold.  In fact, Voisin's method can be utilized precisely because the variation of Hodge structure on a cubic 7-fold behaves just like the variation of Hodge structure of a Calabi-Yau 3-fold.  We explain this relationship concretely using Kontsevitch's noncommutative geometry.  Namely, we show that for a cubic 7-fold, there is a noncommutative CY 3-fold which has an isomorphic Griffiths group.  

Similarly, one can consider other examples of Fano manifolds with with the same type 
of variation of Hodge structure as a Calabi-Yau threefold (FCYs). 
Among the complete intersections in weighted projective spaces, there 
are only three classes of smooth FCY manifolds; the cubic 7-fold $X_3$,
the fivefold quartic double solid $X_4$, and the fivefold intersection 
of a quadric and a cubic $X_{2.3}$.  We settle the two remaining cases, following Voisin's method to demonstrate that the Griffiths group for a general
complete intersection FCY manifolds, $X_4$ and $X_{2.3}$, is also infinitely generated. 

 In the case of $X_4$, we also show that there is a noncommutative CY 3-fold with an isomorphic Griffiths group.  Finally, for $X_{2.3}$ there is a noncommutative CY 3-fold, $\mathcal B$, such that the Griffiths group of $X_{2.3}$ surjects on to the Griffiths group of $\mathcal B$.  We finish by discussing some examples of noncommutative covers which relate our noncommutative CYs back to honest algebraic varieties such as products of elliptic curves and $K3$-surfaces.
\end{abstract}

\section{Introduction}
A fundamental approach to studying subvarieties of an algebraic variety, $X$,  is through the Chow ring, i.e., the ring of all algebraic cycles on $X$ up to rational equivalence with product given by the intersection pairing.  Then again, one can also study this ring up to algebraic equivalence, or homological equivalence for that matter.  One might wonder; what is the difference between these different types of equivalence? 

Well to compare, e.g., algebraic and homological equivalence we may simply study their difference, i.e., the group of algebraic cycles homologically equivalent to zero modulo algebraic equivalence.  This is called the Griffiths group.  The name and the notation for ${\rm Griff}^p(X)$ come from an example due to Griffiths \cite{Gr}, where he famously, ``put an end to the
belief that algebraic and homological equivalence of algebraic cycles 
might coincide.''\footnote{
see  S. Zucker's review of \cite{Cl}, MR720930.}

Griffiths' example was the quintic hypersurface.  Specifically, he showed that for a general quintic hypersurface,  the Griffiths group corresponding to $2$-cycles, ${\rm Griff}^2(X)$, is nonzero.
Moreover, Griffiths demonstrated that  ${\rm Griff}^2(X)$ has an element 
of infinite order.  This example was amplified by Clemens in \cite{Cl}, who showed that the rational Griffiths group, 
${\rm Griff}^2_{\QQ}(X)$
of these quintics has (countably) infinite dimension 
as a vector space over ${\QQ}$.

In a major advance \cite{Voi1}, Voisin reenvisioned Griffiths' example in a much more structured context.  Indeed, in loc.~cit. she recovers Clemens' amplification by providing a general way for both producing non-trivial algebraic cycles and showing that they are linearly independent in the rational Griffiths group.
This general philosophy led to seminal work  in \cite{Voi2} where Voisin proves that such a result is true for a general non-rigid 
Calabi-Yau threefold, i.e., ${\rm Griff}^2_{\QQ}(X)$ is a countably infinite vector space over $\Q$ for such threefolds.

Another example with nontrivial
Griffiths group was provided by Ceresa \cite{Ce}, who proved that 
if $C$ is a generic curve of genus $\ge 3$ embedded by the Abel-Jacobi 
map in its Jacobian, $J(C)$, and if $C^-$ is the image of $C \subseteq J(X)$ 
after the multiplication by $-1$,  
then $C - C^-$ is an element of infinite order in ${\rm Griff}^{g-1}(X)$.
In the special case when $g = 3$ (so that $X = J(C)$ is of dimension 3), 
Nori \cite{No} showed that, once again, ${\rm Griff}^2(X),$ is of infinite dimension over $\QQ$.\footnote{
see also \cite{Scho} for an example of a 3-fold, $X$, with $K_X = 0$ 
and ${\rm Griff}(X)$ of infinite rank over a field $k \not= {\CC}$.
}

A common theme among these examples is that the variety, $X$, has trivial canonical bundle, 
${\rm dim}(X) = n$ is odd the nontrivial Griffiths group which appears is the ``middle'' Griffiths group: $ {\rm Griff}^{\frac{n-1}{2}}(X)$.
On the other hand, utilizing Voisin's method in \cite{Voi1}, Albano and Collino demonstrated that the general cubic $7$-fold, $X_3$, has infinitely generated ${\rm Griff}^{3}(X_3)$ \cite{AC}.  Meanwhile, in \cite{No}, M. Nori constructs a class of Fano varieties with non-trivial (non-middle) 
Griffiths groups ${\rm Griff}^p_{\QQ}(X)$.  

While the example of Albano and Collino is seven as opposed to three-dimensional and Fano as opposed to Calabi-Yau, it actually bears remarkable homological and Hodge-theoretic resemblance to a Calabi-Yau threefold.  
Indeed, as early as the 1980's the cubic sevenfold $X$ was regarded in the physics literature as a mirror of a rigid Calabi-Yau 
threefold with large Picard group, see p. 58-60 in \cite{CHSW}.
In particular, the cubic 7-fold $X$ has a variation of Hodge structure
 (v.H.s.)
similar to that of a non-rigid Calabi-Yau threefold
(see e.g. \cite{CDP}).


One can ask whether there are other examples of higher-dimensional 
manifolds with a Hodge variation similar to that of a  Calabi-Yau 
threefolds.  The answer is a resounding yes.\footnote{The second named author is grateful to Maximilian Kreuzer, who presented him with a list of $3284$ 
such hypersurfaces in weighted projective spaces.}  Concretely, one can define the notion of a manifold of Calabi-Yau type and see such manifolds
manifest as Fano complete intersections in weighted projective space.  However, the restriction of being smooth yields 
only three projective families \cite{Schi, CDP, BBVW,IM}:  
\begin{itemize}

\item  the smooth cubic 7-folds $X_3 \subseteq \PP^8$,

\item the smooth hypersurfaces $X_4$ 
   of degree 4 in the weighted  projective space $\PP^6(1^6;2) = \PP^6(1:1:1:1:1:1:2)$, and 
                  
\item the smooth complete intersections $X_{2.3} \subseteq \PP^7$ 
                       of a quadric and a cubic.

\end{itemize}

Our approach to the study of Griffiths groups is therefore twofold.  First, like Albano and Collino, we employ Voisin's method in \cite{Voi1} to settle the two remaining cases and demonstrate that the rational Griffiths group is infinitely generated for all of the cases above.  Second, in these examples, the Hodge-theoretic comparisons can be categorified by comparing the bounded derived categories of coherent sheaves on these spaces to that of 3-dimensional Calabi-Yau category, or in the language of Kontsevitch, a 3-dimensional Calabi-Yau noncommutative space.   Therefore, we concretely tie these families back to Voisin's theorem by abstracting the situation to the noncommutative setting.  This culminates in the following result:

\begin{theorem}
Suppose $X$ is a smooth Fano-Calabi-Yau complete intersection in weighted projective space.  There is a noncommutative space, $\mathcal A$, and an isomorphism of Griffiths groups,
\[
 {\rm Griff}_{\QQ}(X) =  {\rm Griff}_{\QQ}(\mathcal A).
\]
If $X$ is sufficiently general, then $ {\rm Griff}_{\QQ}(X) =  {\rm Griff}_{\QQ}(\mathcal A)$ is a countably infinite vector space over $\QQ$.    Furthermore, when $X$ is a cubic $7$-fold or a hypersurface of degree $4$ in $\PP^6(1^6;2)$, then  $\mathcal A$ is a $3$-dimensional Calabi-Yau.
In the final case, when $X$ is a smooth complete intersection of a quadric and a cubic  in $\P^7$, there is another $3$-dimensional Calabi-Yau noncommutative space, $\mathcal B$ which is a localization of $\mathcal A$ and   ${\rm Griff}_{\QQ}(X)  =  {\rm Griff}_{\QQ}(\mathcal A)$ surjects onto  $ {\rm Griff}_{\QQ}(\mathcal B)$. 
 \end{theorem}

For certain families, the noncommutative CY 3-folds appearing above can be related back to algebraic varieties with trivial canonical class by utilizing the categorical covering picture from \cite{BFK11}.  For example, if $f(x_0,x_1,x_2), g(x_3,x_4,x_5), h(x_6,x_7,x_8)$ all define smooth elliptic curves, $E_1, E_2, E_3$ respectively, then the product $E_1 \times E_2 \times E_3$ is a $\Z_3 \times \Z_3$-cover of the noncommutative CY 3-fold corresponding to the cubic sevenfold, $X_3$,  defined by $f+g+h$.  From this we are able to obtain an isomorphism:
\[
 {\rm Griff}_{\QQ}(E_1 \times E_2 \times E_3)^{\Z_3 \times \Z_3} \cong {\rm Griff}_{\QQ}(X_3)^{\Z_3 \times \Z_3},
\]
where the ${\Z_3 \times \Z_3}$-action on ${\rm Griff}_{\QQ}(E_1 \times E_2 \times E_3)$ is determined by a correspondence on $E_1 \times E_2 \times E_3 \times E_1 \times E_2 \times E_3$ and the ${\Z_3 \times \Z_3}$-action on ${\rm Griff}_{\QQ}(X_3)$ comes from an easily described subgroup of $\op{PGL}(9)$.

This paper is organized as follows.  In \S\ref{sec: background}, we gather some basic definitions and discuss their relevance in the literature.  In \S\ref{sec: deformations}, we describe to necessary background to implement Voisin's method.  In \S\ref{x4}, we use this method to prove that the Griffiths group for the general smooth hypersurface $X_4 \subseteq \PP^6(1^6;2)$ is infinitely generated.  Similarly in \S\ref{x23}, we do the same for $X_{2.3}$, the general intersection of a quadric and cubic in $\PP^7$.  We then pass to the categorical portion of the paper where, in \S\ref{sec: admissible}, we extend the notion of Griffiths groups to certain types of noncommutative spaces and show that this notion behaves well with respect to semi-orthogonal decompositions.  In \S\ref{sec: CY and FCY}, we apply this formalism to the examples in the above list and compare each of these examples with the Griffiths group of a noncommutative CY 3-fold.  Finally, in \S\ref{sec: categorical covers}, we implement the categorical covering picture in \cite{BFK11}, to establish a connection between these noncommutative CY 3-folds and certain 3-dimensional algebraic varieties with trivial canonical-class.

\vspace{2.5mm}
\noindent \textbf{Acknowledgments:}
The authors owe their sincere gratitude to Pranav Pandit, Maximillian Kreuzer, Matthew Ballard, Bertrand T\"oen, and Maxim Kontsevich for stimulating and extremely useful conversations and would like to thank them all  for their time, patience, and insight. 
The first and third named authors were funded by NSF DMS 0854977 FRG, NSF DMS 0600800, NSF DMS 0652633 FRG, NSF DMS 0854977, NSF DMS 0901330, FWF P 24572 N25, by FWF P20778 and by an ERC Grant.
\vspace{2.5mm}


 \section{Background} \label{sec: background}

\subsection{Algebraic cycles and Griffiths groups}
Let $X$ be a non-singular variety over an algebraically closed field $k$.  Unless otherwise stated,  we assume that $k = {\CC}$. 
Let $Z^p(X)$ be the free Abelian group generated by the irreducible subvarieties 
on $X$ of codimension $p$. 
Let $Z^p_{rat}(X)$ be the set of all cycles, $z \in Z^p(X)$, rationally equivalent to zero,
let $Z^p_{alg}(X)$ be the set of all  
cycles, $z \in Z^p(X)$, algebraically equivalent to $0$, and let $Z^p_{hom}(X)$ be the set of of all cycles, $z \in Z^p(X)$, 
homologically equivalent to $0$.  We have containments,
\[
Z^p_{rat}(X) \subseteq Z^p_{alg}(X) \subseteq Z^p_{hom}(X).
\]
The $p$-th Chow group,
\[
CH^p(X) := Z^p(X) / Z^p_{rat}(X),
\]
 is the quotient group of $Z^p(X)$ 
by rational equivalence of algebraic cycles on $X$.  For ${\rm dim}(X) = n$, 
one can equivalently use the notation $Z_p(X) = Z^{n-p}(X)$ and 
$CH_p(X) = CH^{n-p}(X).$

Similarly let us define the notation,
\[
CH^p_{alg}(X) := Z^p_{alg}(X)/ Z^p_{rat}(X) 
\]
and 
\[
CH^p_{hom}(X) := Z^p_{hom}(X)/ Z^p_{rat}(X). 
\]
The group, $CH^p_{hom}(X)$, can alternatively be described as  the kernel of the cycle class 
map,
\[
\alpha : CH^p(X) \rightarrow H^{2p}(X, {\bf Z}).
\] 
Meanwhile, the subgroup, 
\[
CH^p_{alg}(X) \subseteq CH^p_{hom}(X),
\]
 is a divisible algebraic group.

\begin{definition} The $p$-th \textbf{Griffiths group} of $X$ is the quotient,
\[
{\rm Griff}^p(X) = CH^p_{hom}(X)/CH^p_{alg}(X).
\]
\end{definition}
\noindent  Note that ${\rm Griff}^n(X) = {\rm Griff}^1(X) = 0$. 

The Griffiths group of $X$ can also be realized as
\[
Z^p_{hom}(X) / Z^p_{alg}(X)
\]
which is the kernel of the cycle class homomorphism
\[
c_{alg} : Z^p(X) /  Z^p_{alg}(X) \to H^{2p}(X, \Z)
\]

Let $\op{K}_0(X)$ denote the Grothendieck group of algebraic vector bundles on $X$ and $\op{K}_0^{sst}(X)$ denote the Grothendieck group of algebraic vector bundles modulo algebraic equivalence.
The Chern character map induces a rational isomorphism,
 \begin{equation}
 \op{ch}: \op{K}_0(X) \otimes \Q \to CH^*(X) \otimes \Q,
 \end{equation}
which preserves algebraic equivalence and yields,
\begin{equation} \label{eq: semitopological K-theory}
 \op{ch}_{sst}: \op{K}_0^{sst}(X) \otimes \Q \to Z^p(X) /  Z^p_{alg}(X) \otimes \Q.
 \end{equation}

Hence, the total rational Griffiths group,
\[
{\rm Griff}_{\QQ}(X) := \bigoplus_{p=0}^{\op{dim}(X)} {\rm Griff}^p(X) \otimes \QQ,
\]
is isomorphic to,
\[
\op{ker}(c_{alg} \circ \op{ch}_{sst}).
\]

This will be the starting point for our categorical definition of the total rational Griffiths group.  
Namely, in \S\ref{sec: admissible}, we show that for an admissible subcategory of the bounded derived category of coherent sheaves on $X$,  $\mathcal A \subseteq \dbcoh{X}$,  we can restrict the map $c \circ \op{ch}_{sst}$ to $\op{K}_0^{sst}(\mathcal A)$ and define ${\rm Griff}_{\QQ}(\mathcal A)$ as the kernel of this restriction.







\medskip

\subsubsection{\bf Manifolds of Calabi-Yau type}
\begin{definition}
Let $X$ be a smooth compact complex  variety of odd dimension $2n+1$, $n\ge 1$. 
We call $X$ a {\it generalized Calabi-Yau manifold} if
\begin{enumerate}
\item
the middle Hodge structure is similar to that of a Calabi-Yau threefold, 
i.e.: 
$$h^{n+2,n-1}(X)=1, \qquad and \quad  h^{n+p+1,n-p}(X)=0 \;\;for\;p\ge 2;$$
\item
for any generator $\omega\in H^{n+2,n-1}(X) \cong \CC$, the contraction map 
$$H^1(X,TX)\stackrel{\omega}{\longrightarrow} H^{n-1}(X,\Omega_X^{n+1})$$
is an isomorphism; 
\item
the Hodge numbers $h^{k,0}(X), k = 1,2,...,2n$ are zero. 
\end{enumerate}
\end{definition}
Notice that for $n = 1$ the above definition coincides with the 
definition of a Calabi-Yau threefold.  Also, recent joint work of Manivel and the second author \cite{IM} provides a series of examples of 
manifolds of Calabi-Yau type of dimension $>3$.  These examples come from certain complete intersections in homogeneous varieties, 
starting from hypersurfaces in projective spaces.

The starting point for this definition is the following property described by Donagi and Markman,
which holds for all Calabi-Yau threefolds (see \cite{DM1},\cite{DM2}): 

\medskip

\noindent
{\bf (DM)}  
{\sl The relative intermediate Jacobian forms an integrable system over the gauged\\ 
${}$\hspace{1.3cm} moduli space of any Calabi-Yau threefold.}

\medskip

As remarked in \cite{DM2}, this property cannot 
be generalized for Calabi-Yau varieties $X$ of higher dimension $n \ge 4$.
Indeed, an intermediate Jacobian exists 
only for manifolds of odd dimension, so (DM) cannot even be stated 
correctly for Calabi-Yau manifolds of even dimension.  
On the other hand, in the case when $n > 3$ is odd,
 there is still a natural 2-form on the relative intermediate Jacobian,
$\sigma$, inherited by the cotangent fibration over the moduli space of $X$.
However, while the fibers of the relative intermediate Jacobian are still 
isotropic with respect to $\sigma$, the Yukawa cubic 
(or the Donagi-Markman cubic) on the tangent fibration over the moduli space 
of $X$ vanishes.  Hence $\sigma$ is degenerate over the 
general point, see Remark 7.8 and Theorem 7.9 in \cite{DM2}.    
In contrast, on the gauged moduli spaces of the generalized Calabi-Yau 
(2n+1)-folds as above the relative intermediate Jacobian can form  
an integrable system, see \cite{IM}.  

All known examples of manifolds of Calabi-Yau type are Fano.
Hence, we introduce the following terminology: 
\begin{definition}
A {\it Fano-Calabi-Yau (FCY) manifold} 
is a manifold $X$ of Calabi-Yau type that is Fano, 
i.e., one for which the anticanonical class, $-K_X$, is ample. 
\end{definition}

FCY manifolds are always of odd dimension at least $5$. 
Another common property of FCY manifolds  
and Calabi-Yau manifolds is that the deformations 
of Fano manifolds are not obstructed, see e.g. \cite{Ran}.
Unobsructedness of deformation spaces
is an important property of Calabi-Yau manifolds 
known as the Bogomolov-Tian-Todorov theorem,
see e.g. \cite{Voi4}.   
As a consequence, one can speak about moduli spaces $\cX$ of FCY manifolds $X$ 
and relative Jacobians $\cJ(\cX) \rightarrow \cX$ 
above them.  

\bigskip

\section{Deformations of triples, Noether-Lefschetz loci, and infinitesimal invariants} \label{sec: deformations}

\subsection{Deformations of triples $(\lambda,Y,X)$ and Noether-Lefschetz loci}\label{lyx}
For a FCY manifold $X \subseteq \PP^N$ of dimension $2n+1$,  
denote by $\cX$ its deformation space. Suppose for 
simplicity that ${\rm Pic}\ X = \ZZ H$ where 
$H$ is the hyperplane section. In particular
$-K_X = rH$ for some integer $r = r(X) > 0$, i.e. 
$X$ is a prime Fano manifold of index $r$;  
and suppose in addition that the index $r(X) \ge 2$.
Fix a positive integer $d < r(X)$. Then the smooth 
divisors $Y \in |\cO_X(d)|$ will be Fano manifolds, 
and denote by $\cY$ their deformation space.
\footnote{This deformation space exists by result of Z. Ran \cite{Ran}.}
Let 
\begin{equation}\label{projections}
\cY \stackrel{q}\longleftarrow \cG = \{ (Y,X): Y\subseteq X \} \stackrel{p}\longrightarrow \cX
\end{equation}
with $q$ and $p$ being the restrictions of the natural projections 
of $\cY \times \cX \supseteq \cG$ to $\cY$ and $\cX$. 
Consider the following diagram 

\begin{picture}(100,80)
\put(12,18){\makebox(0,0){0}}
    \put(16,20){\vector(2,1){10}}
\put(10,30){\makebox(0,0){0}}
    \put(15,30){\vector(1,0){10}}
\put(35,30){\makebox(0,0){$T_Y(-X)$}}
    \put(47,30){\vector(1,0){16}}
    \put(44,32){\vector(2,1){10}}
\put(50,50){\makebox(0,0){0}}
    \put(52,48){\vector(1,-1){04}}
\put(62,40){\makebox(0,0){$T_{Y\subseteq X}$}}
    \put(70,43){\vector(2,1){16}}
    \put(66,37){\vector(1,-1){04}}
\put(73,30){\makebox(0,0){$T_X$}}
    \put(78,30){\vector(1,0){10}}
    \put(78,28){\vector(1,-1){10}}
         \put(83,32){\makebox(0,0){\scriptsize $\alpha$}}
\put(95,00){\makebox(0,0){0}}
\put(95,15){\makebox(0,0){$N_{Y|X}$}}
    \put(95,10){\vector(0,-1){06}}
    \put(99,12){\vector(1,-1){06}}
\put(95,30){\makebox(0,0){$T_X|_Y$}}
    \put(102,30){\vector(1,0){10}}
    \put(95,25){\vector(0,-1){06}}
    \put(98,22){\makebox(0,0){\scriptsize $\gamma$}}
\put(95,55){\makebox(0,0){$T_Y$}}
     \put(95,50){\vector(0,-1){14}}
     \put(99,57){\vector(2,1){10}}
         \put(98,43){\makebox(0,0){\scriptsize $\beta$}}
\put(95,72){\makebox(0,0){0}}
    \put(95,67){\vector(0,-1){06}}
\put(108,03){\makebox(0,0){0}}
\put(117,30){\makebox(0,0){0}}
\put(113,63){\makebox(0,0){0}}
\end{picture}

\bigskip

\noindent
in which the vertical and the horizontal row are 
tangent sequence for $Y \subseteq X$ and the adjoint 
sequence for $Y \subseteq X$ twisted by $T_Y$, 
and $T_{Y\subseteq X}$ is the kernel of the composition 
map $\gamma \circ \alpha$. Under the particular assumptions 
as above, there is an identification of the normal bundle,  $N_{Y|X} = \cO_Y(d)$.

Let $\lambda \in H^{n,n}_o(Y,\ZZ) \subseteq H^{2n}(Y)$ be the class 
of a primitive integer $n$-cycle $Z_\lambda$ on $Y$. 
The class $\lambda$ determines locally around $(Y,X) = (Y_o,X_o)$ 
a family $\cF_{\lambda} \subseteq \cG$  
defined by all local deformations $(Y_t,X_t)$ of $(Y,X)$ 
inside $\cG$ for which the class $\lambda \in H^{2n}(Y_t) = H^{2n}(Y)$ 
remains of type $(n,n)$. 

By definition, the Noether-Lefschetz locus $\cF \subseteq \cY$ 
is the set of all $Y$ for which the primitive integer cohomology 
$H^{n,n}_o(Y,\ZZ) \not= 0$;
in particular $\cF_{\lambda}$ is a component of $\cF$.

For fixed $X$, we define the Noether-Lefschetz locus inside $|\cO_X(d)|$,
$\cF(X)$, to be the set of all $Y \in |\cO_X(d)|$ 
that belong to $\cF$; and let 
$\cF(X)_{\lambda} \subseteq \cF(X)$ to be the set of all 
$Y \in |\cO_X(d)|$ that belong to $\cF_{\lambda}$.  

For a given triple, $(\lambda,Y,X)$, 
let $T\cF_\lambda$ be the tangent space to 
$\cF_\lambda \subseteq \cY \times \cX$ 
at $(Y,X) = (Y_o,X_o)$. 
Suppose in addition that the Kodaira-Spencer map 
$$
\rho: H^0(N_{Y|X}) \rightarrow H^1(T_Y)
$$
is injective. 
By \S 1 of \cite{Voi1} 
$$
T\cF_\lambda = \{ (v,u) \in H^1(T_Y) \times H^1(T_X):  \beta(v) = \alpha (u) 
\ \mbox{ and } \ u \bullet \lambda^{n,n} = 0 \}, 
$$ 
where \ $\bullet$ \ is the cup-product 
$$
H^1(T_Y) \otimes H^n(\Omega^n_Y) \rightarrow H^{n+1}(\Omega^{n-1}_Y), \ 
(u,\lambda) \mapsto u \bullet \lambda .
$$
Now consider the following diagram:

\begin{picture}(136,40)


\put(25,30){\makebox(0,0){$H^0(N_{Y|X})$}}
\put(36,30){\vector(1,0){14}}
\put(60,30){\makebox(0,0){$H^1(T_Y)$}}
\put(68,30){\vector(1,0){10}}
\put(88,30){\makebox(0,0){$H^1(T_Y|_X)$}}
\put(99,30){\vector(1,0){6}}
\put(116,30){\makebox(0,0){$H^1(N_{Y|X})$}}

\put(57,10){\makebox(0,0){$H^{n+1}(\Omega^{n-1}_Y)$}}

\put(57,27){\vector(0,-1){13}}
\put(33,27){\vector(1,-1){13}}

\put(32,20){\makebox(0,0){{\scriptsize $\lambda^{n,n}\circ \rho$}}}
\put(62,20){\makebox(0,0){{\scriptsize $\lambda^{n,n}$}}}
\put(43,32){\makebox(0,0){{\scriptsize $\rho$}}}

\end{picture}

\noindent
The assumptions on $X$ and $Y$ guarantee that
\[
H^1(N_{Y|X}) = H^1(\cO_Y(d)) = H^1(K_Y \otimes \cO(r)) = 0
\]
by the Kodaira vanishing theorem.
Therefore, we can apply the argument in \S 2 from \cite{AC} to conclude:

\begin{proposition}\label{lyx-nl}
Let $(\lambda,Y,X)$ be as above, let 
$p_*: T\cF_{\lambda} \rightarrow T{\cX}$ 
be the map induced by the projection $p: \cF_{\lambda} \rightarrow \cX$,
and suppose that the composition
$$
\lambda^{n,n}\circ \rho : H^0(N_{Y|X}) \rightarrow H^{n+1}(\Omega^{n-1}_Y)
$$
is an isomorphism. 
Then 

\begin{enumerate}

\item The map $p_*: T\cF_\lambda|_{Y,X} \rightarrow T{\cX}|_{X} = H^1(T_X)$   
is also an isomorphism, and hence the family $\cF_{\lambda}$ 
is smooth of codimension $h^{n+1,n-1}(Y)$ in $\cG$ at $(Y,X)$, 
and the projection $p: \cF_\lambda \rightarrow \cX$ is an isomorphism  
over a neighborhood of $X$. 

\item There are infinitely many 0-dimensional components of the Noether-Lefschetz locus 
$\cF(X) \subseteq |\cO_X(d)|$, forming a countable 
subseteq in $|\cO_X(d)|$.

\end{enumerate}

\end{proposition}

\begin{remark}
Part (1) is the analog of Lemma 2.3 and Proposition 2.4 from \cite{AC},
which in turn reproduces the original argument of Voisin in \S1 of \cite{Voi1}.
Part (2) follows from (1) based upon an argument due originally 
to M. Green.  
By (1), the set $\cF(X)_\lambda$ is reduced and 0-dimensional.
In particular, the Noether-Lefschetz locus $\cF(X)$ in $|\cO_X(d)|$ 
has at least one 0-dimensional component. By an argument due originally 
to M. Green the latter implies that $\cF(X)$ has countably many 
$0$-dimensional components and they form a dense subset of $|\cO_X(d)|$, 
see the proof of Proposition 1.2.3 in \cite{B-MS} or Propositions 2.4 and 2.5 
in \cite{AC} together with the references found therein. 
\end{remark}




In \S\ref{x4} and \S\ref{x23} we study two examples of triples, $(\lambda,Y,X)$, 
that fulfill the conditions of the above proposition. In order to verify 
these conditions, we follow the approach used initially by C. Voisin in 
\cite{Voi1}, i.e., we verify these conditions using the graded 
rings of $X$ and $Y$. 




\medskip

\subsection{The infinitesimal invariant of a normal function associated 
to a deformation of a triple $(\lambda,Y,X)$}\label{lyxd}

Let $(\lambda,Y,X)$ be a triple which fulfills the conditions 
of Proposition \ref{lyx-nl}, and suppose 
further that the Hodge conjecture holds
for $Y$. 
Then by (1) of Proposition \ref{lyx-nl} the family, $\cF_\lambda$, is isomorphic to $\cX$ over a neighborhood, $\cU$
of 
$X = X_o \in \cX$. Since the argument is local, we can suppose that 
$\cU = \cX$. 
Thus for any $X_t \in \cX$ the class $\lambda$ 
is represented by an algebraic $n$-cycle $Z_{\lambda,t}$ on $Y_t \subseteq X_t$.
Since $\lambda$ is, by assumption, homologically equivalent to zero,
$Z_{\lambda,t} = \partial \Gamma_{\lambda,t}$ 
for some real $(2n+1)$-chain $\Gamma_{\lambda,t}$ on $X$. 

Let $\cH^{2n+1}_\cX \rightarrow \cX$ 
be the $(2n+1)^{\op{th}}$ cohomology 
bundle and $\cH^{i,j}_\cX \rightarrow \cX$, $i+j = n+1$  be its Hodge subbundles over $\cX$, with the holomorphic filtration 
$F^i\cH^{2n+1}_\cX = \oplus_{k \ge i} \cH^{k,2n+1-k}_\cX$. 
Let 
$$
\cJ_\cX = (\cF^{n+1}\cH^{2n+1}_\cX)^\vee /\cH_{2n+1}(\cX,\ZZ)  \ \longrightarrow \cX
$$ 
be the intermediate Jacobian bundle 
over $\cX$. The cycles $Z_{\lambda,t}, \ t \in \cX^o$ define a normal 
function
$$
\nu_{\lambda}: \cX^o \rightarrow \cJ(\cX^o), 
 \ \ X_t \longmapsto \nu_\lambda(t) := \Phi_{t}(Z_{\lambda,t}),  
$$ 
where $\Phi_t$ is the Abel-Jacobi map for $X_t$, see \S 7 Ch.II in Vol.II of \cite{Voi3}.
The
normal function $\nu_\lambda$ lifts to a holomorphic section 
$\psi_\lambda$ of $(\cF^{n+1}\cH^{2n+1}_\cX)^\vee$, defined on the sections 
$\omega_t$ of $\cF^{n+1}\cH^{2n+1}_\cX$
by
$$
\psi_\lambda(t)(\omega_t) = \int_{\Gamma_t} \omega_t,    
$$
where $\partial \Gamma_t = Z_{\lambda,t}$ -- see above.


Now we use the assumption that $X$ is a FCY manifold of dimension $2n+1$, 
i.e., the only nonzero 
middle Hodge numbers of $X$ are $h^{n+2,n-1} = h^{n-1,n+2} = 1$ and 
$h^{n+1,n} = h^{n,n+1}$. From this fact, it follows that the variation of Hodge structure is a map,
$$
\overline{\nabla}: \cH^{n+1,n}_{\cX} \otimes T_\cX \longrightarrow \cH^{n,n+1}_{\cX}.
$$  
Let ${\rm Ker}\overline{\nabla}$ be the kernel bundle of $\overline{\nabla}$.

The Griffiths infinitesimal invariant $\delta\nu_\lambda$ of the normal 
function $\nu_\lambda$ is a section of the dual bundle, $({\rm Ker}\overline{\nabla})^\vee$, 
defined as follows. 
Let $\sum_i \ \omega_i \otimes \chi_i \in {\rm ker}\overline{\nabla}$, 
and let 
$\tilde{\omega}_i(t)$ be sections of $\cF^{n+1}\cH^{2n+1}_\cX$
such that $\tilde{\omega}_i(0) = \omega_i$. Then
\begin{equation}\label{delta(nu)_Griffiths}
\delta\nu_\lambda (\sum_i \omega_i \otimes \chi_i) 
=
\sum_i \chi_i(\psi_\lambda(\tilde{\omega}_i) 
- \psi_\lambda(0)(\sum_i \nabla_{\chi_i}(\tilde{\omega}_i)),
\end{equation}
see p.721-722 in \cite{AC}, or \cite{Gr} and \cite{Voi1}. 

In \S\ref{x4} and \S\ref{x23} we show that the infinitesimal invariants 
$\delta{\nu}_\lambda$ of certain special primitive algebraic Hodge classes, 
$\lambda$, on two given FCY manifolds $X$ are non-zero. 
This fact together with the following lemma will provide us with the nontriviality of ${\rm Griff}_\QQ(X)$.  Later down the road, the fact that ${\rm Griff}_\QQ(X)$ is infinitely generated will follow from this as well.

\begin{lemma}\label{dnu-non-zero}
Let $X$, $Y$, $\lambda$, and $Z_\lambda$ be as above.
If $\delta{\nu}_\lambda \not= 0$ then for the general 
$X_t \in \cX$ the algebraic $n$-cycle $Z_{\lambda,t}$ 
represents a non-torsion element of ${\rm Griff}(X_t)$. 
\end{lemma}

\begin{proof}
See \cite{Voi1} or \cite{B-MS}. 
\end{proof}


\bigskip

\section{The Griffiths group of the 5-fold quartic double solid $X_4$}\label{x4}

\subsection{The 5-fold quartic double solid and its quadratic sections}

A 5-fold quartic double solid is a double covering 
$$
\pi:X \rightarrow \PP^5
$$ 
branched over a quartic hypersurface $B \subseteq \PP^5$. 
It can be represented as a hypersurface $X = X_4$ of degree $4$ 
in the weighted projective space $\PP^6(1^6;2) = \PP^6(1:1:1:1:1:1:2)$. 
If the opposite is not explicitly stated we assume that $X$ is smooth, 
which is equivalent to the smoothness of its branch locus $B$.

If $(x;y) = (x_0:...:x_6:y)$ are the coordinates in $\PP^6(1^6;2) = \PP^6(x;y)$
and $B = (f_4(x) = 0)$ is the equation of $B \subseteq \PP^5 = \PP^5(x)$ then 
the equation of $X = X_4 \subseteq \PP^6(x;y)$ is 
$$
y^2 - f_4(x) = 0.
$$  
In turn, any smooth hypersurface $X_4 \subseteq \PP^6(1^6;2)$ 
which does not contain the point $(0;1)$ is equal to a 5-fold quartic double 
solid over $\PP^5$. To see this, let 
$$
f(x;z) = z^2 + 2q(x)z + r(x) = 0,
$$ 
${\rm deg}\ q(x) = 2, {\rm deg}\ r(x) = 4$ 
be the equation of $X = X_4 \subseteq \PP^6(x;z) = \PP^6(1^6;2)$.
Then $f(x;z) = (z + q(x))^2 - (q(x)^2 - r(x))$, and after changing 
the weight 2 variable $z$ by $y = z + q(x)$ the equation of 
$X$ becomes $y^2 - f_4(x) = 0$, where 
$$
f_4(x) = q(x)^2 - r(x).
$$
Clearly the map $(x;y) \mapsto X$ restricts to a double covering 
$
\pi: X \rightarrow \PP^5 = \PP^5(x), 
$
with branch locus equal to the quartic 
hypersurface, $B$, defined by the equation $f_4(x) = 0$.




\medskip

\subsection{The graded ring of $X_4 \subseteq \PP^6(1^6;2)$ 
            and of its quadratic sections}\label{rxy-1}

Let 
$$
f(x;y) = y^2 - f_4(x)
$$
be the equation of $X = X_4$ in $\PP^6(1^6;2)$.
In the graded algebra, 
$$
S(X) = \CC[x;y] = \CC[x_0,...,x_5,y],
$$
with ${\rm deg}\ x_i = 1$, ${\rm deg}\ y = 2$ the graded Jacobian 
ideal $J(X)$ of $X$ is generated by the partials 
${\partial f(x,y)}/{\partial x_i} = -{\partial f_4(x)}/{\partial x_i}$
and  ${\partial f(x,y)}/{\partial y} = 2y$.
Let 
$$
R(X) = S(X)/J(X) = \bigoplus_d R_d(X)
$$ 
be the graded Jacobian ring of $X$.
By \cite{Na} 
$$
H_o^{5-p,p}(X) \cong R_{4p-4}(X), \mbox{ for } p = 0,...,5, 
$$
which yields 
$$
h^{5,0}(X) = h^{0,5}(X) = 0, \
h^{4,1}(X) = h^{1,4}(X) = 1, \ 
h^{3,2}(X) = h^{2,3}(X) = 90.
$$

Let $X \subseteq \PP^6(1^6;2)$ be given by $f(x;y) = y^2 - f_4(x) = 0$ as above.
A quadratic section $Y \subseteq X$ which does not contain 
the point $(0;1)$ is given inside $X$ by an equation 
$$
q(x;y) = y - f_2(x) = 0.
$$

\begin{lemma}\label{f22-f4}
If the quadratic section $Y$ of $X$ is as above, 
then the rational projection 
$$
\PP^6(1^6;2) = \PP^6(x;y) \rightarrow \PP^5(x),\ \ (x;y) \mapsto (x)
$$
sends $Y$ isomorphically onto the quartic hypersurface 
$Y_4 \subseteq \PP^5(x)$ defined by the equation 
$$
f(x) = f_2(x)^2 - f_4(x) = 0.
$$
\end{lemma}

\begin{proof}
In the first equation $y^2 - f_4(x)$ of $Y$ replace $y$ by $f_2(x)$ 
coming from the second equation $y - f_2(x) = 0$.  
\end{proof}

Via the interpretation of $Y$ as a quartic hypersurface 
$f(x) = f_2(x)^2 - f_4(x) = 0$ in $\PP^5$, 
its middle primitive cohomology can be computed 
by the formulas from \cite{Na}:
$$
H_o^{4-p,p}(Y) = R_{4p-2}(Y), \ p = 0,...,4,
$$
where 
$R = S(Y)/J(Y) 
= \CC[x_0,...,x_5]/(\frac{\partial f}{\partial x_0}, ..., \frac{\partial f}{\partial x_5})$
is the graded Jacobian ring of $Y$. 
This yields 
$$
h^{4,0}(Y) = h^{0,4}(Y) = 0, \ 
h^{3,1}(Y) = h^{1,3}(Y) = 21, \
h_o^{2,2}(Y) = 141.
$$ 

\medskip

\subsection{Cycles on quadratic sections $Y \subseteq X$}\label{yx}

Let $Y$ be a smooth quadratic section of $X$ which does not 
contain the point $(0;1)$. 
By the discussion from the preceding section, $Y$ is isomorphic to a quartic fourfold $Y_4$.
Therefore by \cite{CM}, the Hodge conjecture holds for $Y$. 
Let $\lambda \in H^{2,2}_o(Y,\ZZ)$ be a Hodge class on $Y$, 
representing a primitive algebraic 2-cycle $Z_\lambda \subseteq Y \subseteq X$, 
see \S\ref{lyxd}.
For the given triple $(\lambda,Y,X)$ the $2^{\op{nd}}$ diagram from \S\ref{lyx}  
becomes

\begin{picture}(136,33)


\put(27,25){\makebox(0,0){$H^0(\cO_Y(2))$}}
\put(40,25){\vector(1,0){11}}
\put(60,25){\makebox(0,0){$H^1(T_Y)$}}
\put(69,25){\vector(1,0){6}}
\put(86,25){\makebox(0,0){$H^1(T_X|_Y)$}}
\put(96,25){\vector(1,0){6}}
\put(118,25){\makebox(0,0){$H^1(\cO_Y(2)) = 0$}}

\put(60,5){\makebox(0,0){$H^3(\Omega^1_Y)$}}

\put(60,22){\vector(0,-1){13}}
\put(40,22){\vector(1,-1){13}}

\put(39,15){\makebox(0,0){{\scriptsize $\lambda^{2,2}\circ \rho$}}}
\put(65,15){\makebox(0,0){{\scriptsize $\lambda^{2,2}$}}}
\put(45,27){\makebox(0,0){{\scriptsize $\rho$}}}

\end{picture}

Below we  translate the above diagram into the language of the graded ring 
$R(Y)$.
In order to do so, we will need to use the following identities that can either be obtained 
 directly, by using the adjoint and the tangent sequence for 
$Y \cong Y_4 \subseteq \PP^5$, or by the Griffiths residue calculus.
\begin{equation}\label{RY2}
H^1(\cO_Y(2)) \cong R_2(Y) \ , \ \ \ \ H^1(T_Y) \cong R_4(Y)
\end{equation}
The equation above, together with the identifications from \ref{rxy-1},  allows us to rewrite the diagram above as,

\begin{picture}(136,33)


\put(37,25){\makebox(0,0){$R_2(Y)$}}
\put(45,25){\vector(1,0){11}}
\put(65,25){\makebox(0,0){$R_4(Y)$}}
\put(74,25){\vector(1,0){6}}
\put(91,25){\makebox(0,0){$H^1(T_X|_Y)$}}
\put(101,25){\vector(1,0){6}}
\put(111,25){\makebox(0,0){$0$}}

\put(65,5){\makebox(0,0){$R_{10}(Y)$}}

\put(65,22){\vector(0,-1){13}}
\put(45,22){\vector(1,-1){13}}

\put(45,15){\makebox(0,0){{\scriptsize $P_\lambda\circ e$}}}
\put(69,15){\makebox(0,0){{\scriptsize $P_\lambda$}}}
\put(50,27){\makebox(0,0){{\scriptsize $e$}}}

\end{picture}

\noindent
where 
$P_\lambda$ is, by slight abuse of notation, multiplication by the polynomial class
$P_{\lambda}$ corresponding 
to $\lambda^{2,2} \in H^{2,2}_o(Y) \cong R_6(Y)$,
and $e$ is multiplication by the polynomial class 
$e = f_2(x) \in R_2(Y)$. 

\medskip

\subsection{The infinitesimal invariant 
            for $(\lambda,Y,X)$ in terms of $R(Y)$}\label{dry-1} 









Let $(\lambda,Y,X)$ be as above, 
$$
\overline{\nabla}: H^{3,2}(X) \otimes H^1(T_X) \rightarrow H^{2,3}(X)
$$
be a variation of Hodge structure for $X$, 
and $\delta{\nu}_\lambda \in ({\rm Ker} \overline{\nabla})^\vee$
%
be the infinitesimal invariant of ${\nu}_\lambda$, \ see \ref{lyxd}. 

Since $X$ is a Fano-Calabi-Yau manifold, the cup-product 
with the unique form  (modulo $\CC^*$), $\omega^{4,1}$,  on $X$  
defines an isomorphism,
$$
{\omega}: H^1(T_X) = H^{-1,1}(X) \stackrel{\sim}\longrightarrow H^{3,2}(X).
$$ 

By the identifications $H^{3,2}(X) \cong R_{4}(X)$ and $H^{2,3}(X) = R_{8}(X)$
from \S\ref{rxy-1}, the v.H.s. $\overline{\nabla}$
is identified with a mapping  
$$
\mu_X: R_{4}(X) \otimes R_{4}(X) \rightarrow R_{8}(X).
$$
It follows from \cite{Voi1}, that in the situation above,
describing the deformations of a triple, $(\lambda,Y,X)$, the following takes place:

\begin{lemma}
The map 
$$
\mu_X: R_4(X) \otimes R_4(X) \rightarrow R_8(X)
$$  
is induced by multiplication of monomials in the homogeneous 
graded ring,
\[
S(X) = \CC[x;y] = \CC[x_0,...,x_7,y].
\]
\end{lemma} 




Again by \cite{Voi1}, under certain conditions  
the infinitesimal invariant $\delta_{\lambda}\nu$ 
can be regarded as a linear form on the kernel of the multiplication 
map 
$$
\mu_Y: R_4(Y) \otimes R_4(Y) \rightarrow R_{8}(Y).
$$ 
More precisely, let  
$y = f_2(x)$ be the equation of $Y$ in $X = (y^2 = f_4(x)) \subseteq \PP^6(x;y)$,
and $f(x) = f_2(x)^2 - f_4(x) = 0$ be the equation of $Y$ in $\PP^5(x)$,
representing $Y$ as a quartic hypersurface in $\PP^5 = \PP^5(x)$.
By the isomorphism, $H^{2,2}_o(Y) \stackrel{\sim}\longrightarrow R_6(Y)$, 
the class $\lambda$ corresponds to $P_{\lambda} \in R_6(Y)$.
Let $e \in R_2(Y)$ be the class defined by the quadric form 
$f_2(x)$. Then the following takes place (see \cite{Voi1} or \cite{AC}):

\begin{lemma}\label{voisin-1}
If the multiplication by $P_{\lambda}.e$ induces an isomorphism 
$$
f_{\lambda} = R_2(Y) \rightarrow R_{10}(Y), 
$$ 
then for any $\omega = \Sigma\ Q_i \otimes R_i \in ker(\mu_Y)$, 
we have the following equality for the infinitesimal invariant:

$$
\delta{\nu}_{\lambda}(\Sigma\ Q_i \otimes R_i) 
= 
\Sigma\ P_{\lambda}Q_i(f_{\lambda}^{-1}(P_{\lambda}R_i)). 
$$
\end{lemma}




\medskip

\subsection{Cycles on the Fermat quartic fourfold} 
Let $Y$ be the Fermat quartic fourfold, i.e., the quadratic section of $X$ as in Lemma \ref{f22-f4} 
given by the equation,
$$
f(x) = f_2(x)^2 - f_4(x) = x_0^4 + \ ... \ + x_5^4.
$$

Let 
$$
R(Y) = S(Y)/J(Y) = \CC[x_0,...,x_5]/(x_0^3,....,x_5^3)
= \bigoplus_{d\ge0} R_d(Y)
$$ 
be the graded Jacobian ring of $Y$. 
By \S\ref{rxy-1}, the primitive cohomology satisfies 
$
H^{2,2}_o(Y) \cong R_6(Y).
$
Following \cite{Shi}, we now describe the rational cohomology classes 
in $H^{2,2}_o(Y)$ and their corresponding elements from $R_6(Y)$. 

Let $\mu_4$ be the group of 4-th roots $\zeta_i$ of unity, and let 
$G = (\mu_4)^6/\Delta$, where $\Delta$ is the diagonal subgroup. 
If $\ZZ_4 = \ZZ/4\ZZ$, then the character group 
$\hat{G}$ is naturally embedded in $(\ZZ_4)^5$ as 
$$
\hat{G} \stackrel{\sim}\longrightarrow \{ \alpha = (a_0,...,a_5): a_0+...+a_5 = 0 \} 
\subseteq (\ZZ_4)^6;
$$
the character $\alpha \in \hat{G} = \op{Hom}(G,\CC^*)$ representing $(a_0,...,a_5)$ sends 
the element $[\zeta_0,...,\zeta_5] \in G = (\mu_4)^6/\Delta$  
\ to \ $\alpha([\zeta_0,...,\zeta_5]) = \zeta_0^{a_0}...\zeta_5^{a_5}$.
Let 
$$
\hat{G}^* = \{ \alpha = (a_0,...,a_5) \in \hat{G}: a_i \not=0, i = 0,...,5 \}.
$$

For $\alpha = (a_0,...,a_5) \in \hat{G}^*$, define its norm 
$$
|\alpha| = \frac{<a_0> + ... + <a_5>}{4},
$$ 
where $<a_i>$ is the unique integer between $1$ and $3$ 
congruent to $a_i$ modulo $4$. 
The natural action 
$$
g = [\zeta_0,...,\zeta_5]: (x_0:...:x_5) \mapsto (\zeta_0x_0,...,\zeta_5x_5)
$$
of $G$ on $\PP^5$ restricts to an action of $G$ on the Fermat quartic $Y \subseteq \PP^5$; 
which in turn induces a representation $g^*$  of $G$ on the primitive cohomology 
group $H^4_o(Y,\CC)$. 
Let 
$$
V_{\alpha} = \{ \lambda \in H^4_o(Y,\CC): g^*(\lambda) = \alpha(g)\lambda \},
$$  
be the eigenspaces of $g^*$ in $H^4_o(Y,\CC)$
defined by the characters $\alpha \in \hat{G}$.  

With the above notation, the results in \cite{Shi} yield the following:

\begin{enumerate}
\item The primitive cohomology obeys the identity, 
$$
H^{2,2}_o(Y,\QQ)\otimes_\QQ \CC = \bigoplus_{\alpha \in B} V_{\alpha}, 
$$  
where 
$
B = \{ \alpha \in \hat{G}^* : |\alpha| = |3.\alpha| = 3 \}
$
and $3.\alpha =3.(a_0,...,a_5) = (3a_0,...,3a_5)$. 

\item Let $C_o(Y)_\QQ$ denotes the subspace of $H^4_o(Y,\QQ)$ spanned by 
classes of primitive algebraic 2-cycles on $Y$.  There is an equality, 
$$
C_o(Y) = C_o(Y)\otimes_\QQ \CC = H^{2,2}_o(Y,\QQ)\otimes_\QQ \CC.
$$
\end{enumerate}

\begin{remark}
In general, the space $C_o(Y)$ spanned by classes of primitive algebraic 2-cycles 
on a fourfold $Y$ is a subspace of $H^{2,2}_o(Y,\QQ)\otimes_\QQ \CC$; the 
coincidence (2) is the statement of the Hodge conjecture for the Fermat 
quartic fourfold, see \cite{Shi}. 
\end{remark}

\medskip

\subsection{The isomorphism $H^{2,2}_o(Y) \rightarrow R_6(Y)$ in coordinates}

Let $Y = (x_0^4 + \ ... \ x_5^4 = 0)$ be the Fermat quartic fourfold. 
Since the graded ring of $Y$ is 
$$
R(Y) = S(Y)/J(Y) = \CC[x_0,...,x_5]/(x_0^3,...,x_5^3)
$$ 
then in a monomial 
$x_0^{b_0}...x_5^{b_5}$, representing a non-zero class 
modulo $J(Y)$, the coordinates $x_i$ can enter only with degrees 
$0$, $1$ and $2$. 
We therefore use the following terminology; we call a nonzero monomial 
$f = x_0^{b_0}x_1^{b_1}...x_5^{b_5}$ a monomial of type $(2^p1^q)$
if $b_i=2$ for $p$ distinct values of $i$ and $b_i=1$ for $q$ distinct values of $i$.  
Now it is easy to see that $R_6(X)$ is generated by the following 141 monomials, 
regarded as classes modulo $J(Y) = (x_0^3,...,x_5^3)$: 

\medskip

{}\hspace{4cm}   $20$ monomials $f$ of type $(222)$;

{}\hspace{4cm}   $90$ monomials $f$ of type $(2211)$;

{}\hspace{4cm}   $30$ monomials $f$ of type $(21111)$;

{}\hspace{4cm}    $1$ monomial  $f$ of type $(111111)$.

\medskip

By definition, $\alpha = (a_0,...,a_5) \in B$ 
iff $|\alpha| = |3.\alpha| = 3$. Since $B \subseteq \hat{G}^*$,
the coordinates $a_i$ take values $k = 1, 2$ and $3$. 
For an element $\alpha = (a_0,...,a_5) \in \hat{G}^*$,
let  
$$
d_k(\alpha) = \#\{ i : a_i = k \}, 
$$
be the number of occurrences of the number $k \in \{ 1,2,3 \}$ 
among the coordinates $a_i$ of $\alpha$.
We call $\alpha$ an element of type $(3^{p}2^{q}1^{r})$ 
if $d_3(\alpha) = p$, $d_2(\alpha) = q$ and $d_1(\alpha) = r$. 
As in \cite{AC}, the isomorphism 
$$
j: H_o^{2,2}(Y) \rightarrow R_6(Y)
$$
is given by:
$$
j:\alpha = (a_0,...,a_5) \longmapsto x_o^{a_0-1}x_1^{a_1-1}...x_5^{a_5-1}.
$$

Now, by 
a simple combinatorial check, we describe all possible $\alpha \in B$
and their corresponding monomials by $j$ as follows:

\begin{itemize}

\item  $20$ elements $\alpha $ of type $(333111)$ 
                $\stackrel{j}\longrightarrow$  
                the $20$ monomials of type $(222)$;

\item $90$ elements of $\alpha $ of type $(332211)$ 
                $\stackrel{j}\longrightarrow$  
                the $90$ monomials of type $(2211)$;

\item  $30$ elements of $\alpha$ of type $(322221)$ 
                $\stackrel{j}\longrightarrow$  
                the $30$ monomials of type $(21111)$;

\item  $1$ element $\alpha$ of type $(222222)$ 
                $\stackrel{j}\longrightarrow$  
                the unique monomial of type $(111111)$.

\end{itemize}

Recall that by (1) the 1-dimensional eigenspaces $V_{\alpha}$ 
of the above $141$ characters  $\alpha$ span the space 
of primitive cohomology $H_o^{2,2}(Y)$.




\medskip

\subsection{Infinite generation of the Griffiths group of $X_4$.} \label{sec: Fermat}

Let $Y$ be the Fermat hypersurface in $\PP^5(x_0:...:x_5)$
defined by 
$$
f(x) = f_2(x)^2 - f_4(x) = x_0^4 + ... + x_5^4 = 0.
$$
Then $R(Y) = \CC[x_0:...x_4]/I(Y)$, where 
$I(Y) = (x_0^3,...,x_5^3)$.

To simplify the notation, for $0 \le i \le j \le ... \le k$ we write 
$$
x_{ij...k} := x_ix_j...x_k
$$ 
for both the monomial and its class in $R(Y)$.
For example $x_{001123} = x_0^2x_1^2x_2x_3 \in R(Y)$.

We call two monomials $x_{ij...k}$ and $x_{i'j'...k'}$ 
{\it dual} 
if \  $x_{ij...k} \ . \ x_{i'j'...k'} = x_{0011...55}$;
for the dual monomial of $x_{ij...k}$ we shall use the 
notation $x_{\widehat{ij...k}}$, i.e. 
$$
x_{\widehat{ij...k}} \ . \ x_{ij...k} = x_{0011...55}.
$$

Let us first find monomials $P_{\lambda}$ and $e$ 
which fulfill the conditions of Lemma \ref{voisin-1}.
To this end, let 
$$
P_{\lambda} = P_{\lambda_{a,b}} = a(x_{001123} + x_{234455}) + b(x_{012233} +x_{014455})
$$
and 
$$
e = ux_{01} + vx_{23} + wx_{45}.
$$
Then 
$$
P_{\lambda_{a,b}}.e = a(ux_{\widehat{0123}}+ vx_{\widehat{0011}} 
                  + vx_{\widehat{4455}} + wx_{\widehat{2345}})
$$

We verify that for generic $a$ and $b$ the linear map 
$$
P_{\lambda_{a,b}}.e: R_2(Y) \rightarrow R_8(Y)
$$ 
is an isomorphism.
In bases $x_{ij}$ and $x_{\widehat{ij}}$  of $R_2(Y)$ and $R_8(Y)$ 
respectively, $P_{\lambda_{a,b}}.e$ acts as follows:
\begin{eqnarray*}
x_{00} \mapsto avx_{\widehat{11}}, & x_{11} \mapsto avx_{\widehat{00}}\\
x_{22} \mapsto bux_{\widehat{33}}, & x_{33} \mapsto bux_{\widehat{22}}\\
x_{44} \mapsto (av+bu)x_{\widehat{55}}, & x_{55} \mapsto (av+bu)x_{\widehat{44}}
\end{eqnarray*}
\begin{eqnarray*}
x_{04} \mapsto bwx_{\widehat{15}}, & x_{15} \mapsto bwx_{\widehat{04}}\\
x_{05} \mapsto bwx_{\widehat{14}}, & x_{14} \mapsto bwx_{\widehat{05}}\\
x_{24} \mapsto awx_{\widehat{35}}, & x_{35} \mapsto awx_{\widehat{24}}\\
x_{25} \mapsto awx_{\widehat{34}}, & x_{34} \mapsto awx_{\widehat{25}}\\
x_{02} \mapsto (au+bv)x_{\widehat{13}}, & x_{13} \mapsto (au+bv)x_{\widehat{02}}\\
x_{03} \mapsto (au+bv)x_{\widehat{12}}, & x_{12} \mapsto (au+bv)x_{\widehat{03}}
\end{eqnarray*}
\begin{eqnarray*}
x_{01} \mapsto avx_{\widehat{01}} + (au+bv)x_{\widehat{23}} + bwx_{\widehat{45}}\\
x_{23} \mapsto (au+bv)x_{\widehat{01}} + bux_{\widehat{23}} + awx_{\widehat{45}}\\
x_{01} \mapsto bwx_{\widehat{01}} + awx_{\widehat{23}} + (av+bu)x_{\widehat{45}}\\
\end{eqnarray*}
Therefore the matrix of $P_{\lambda_{a,b}}.e$ is 
$$
M_{a,b} \ = \ 
\left( \begin{array}{cc}
0  & av \\ 
av & 0
\end{array} \right)
\oplus 
\left( \begin{array}{cc}
0  & bu \\ 
bu & 0
\end{array} \right)
\oplus
{\left( \begin{array}{cc}
0  & aw \\ 
aw & 0
\end{array} \right)}^{\oplus 2}
\oplus 
{\left( \begin{array}{cc}
0  & bw \\ 
bw & 0
\end{array} \right)}^{\oplus 2}
\oplus 
$$
$$
\oplus \ 
{\left( \begin{array}{cc}
0     & av+bu \\ 
av+bu & 0
\end{array} \right)}
\oplus
{\left( \begin{array}{cc}
0     & au+bv \\ 
au+bv & 0
\end{array} \right)}^{\oplus 2}
\oplus \ A_{a,b} \ ,
$$
where 
$$
A_{a,b} \ = \ 
{\left( \begin{array}{ccc}
av    & au+bv & bw\\ 
au+bv & bu    & aw\\
bw    & aw    & av+bu
\end{array} \right)}.
$$

The determinant 
$$
{\rm det} M_{a,b} = a^6b^6u^2v^2w^8(au+bv)^4(av+bu)^2{\rm det} A_{a,b},
$$
where 
$$
{\rm det} A_{a,b} = -a^2bu^3 -(ab^2+a^3)u^2v - (a^2b+b^3)uv^2 - ab^2v^3 
              + (2a^2b-b^3)uw^2 + (2ab^2-a^3)vw^2.
$$

Therefore if 
$$
abuvw(au+bv)(av+bu){\rm det}A_{a,b} \not=0
$$ 
then
$$
f_\lambda = P_{\lambda_{a,b}}.e: R_2(Y) \rightarrow R_8(Y)
$$ 
is an isomorphism, 
and we can apply Voisin's formula from Lemma \ref{voisin-1} 
to compute the infinitesimal invariant  $\delta_{\lambda_{a,b}}$. 
We take 
$$b = 1, u = v = 1 \ \ {\rm and} \ \ w = h
$$ 
where $h$ is a transcendental number.
It follows from the preceding discussion that
$$
f_a := P_{a,1}.e : R_2(Y) \rightarrow R_8(Y)
$$ 
is an isomorphism if $a(a+1) \not= 0$ (since $h$ is transcendental and $a$ is rational and not equal to $-1$, 
the determinant $det(A_{a,1}) = -(a+1)(a^2+a+1)+(a^2-3a+1)h^2$ 
is always nonzero).

Since $f_a = P_{a,1}.e$ is an isomorphism, we can apply the formula from 
Lemma \ref{voisin-1} to evaluate the infinitesimal invariant 
$\delta_{\lambda_{a,1}}$ at the elements of ${\rm Ker}(\mu_Y)$.
The goal is to find elements $v \in {\rm Ker}(\mu_Y)$
such that $\delta_{\lambda_{a,1}}(v) \not= 0$, which by 
Lemma \ref{dnu-non-zero} to see that $\lambda$ 
is an element of infinite order in ${\rm Griff}(X)$,  
cf. \S 4 of \cite{AC}.

Let $Q = x_{2233}$ and $R = x_{0123}$. Then $v = Q \otimes R \in {\rm Ker}(\mu_Y)$, 
and by Lemma \ref{voisin-1} 
$$
\delta_{\lambda_{a,b}}(Q \otimes R) = P_{\lambda_{a,1}}Q(f_a^{-1}(P_{\lambda_{a,1}}R)),
$$
where $f_a^{-1}: R_8(Y) \rightarrow R_2(Y)$ is the inverse to isomorphism 
$f_a = P_{\lambda_{a,1}}$. 
Since 
$$
P_{\lambda_{a,1}}R = (ax_{001123} + ax_{234455} + x_{012233} + x_{014455})x_{0123}
                   = ax_{\widehat{01}} + x_{\widehat{23}}
$$                   
then 
$$
f_a^{-1}(P_{\lambda_{a,1}}R) = f_a^{-1}(ax_{\widehat{01}} + x_{\widehat{23}}) 
= A_{a,1}^{-1}(ax_{\widehat{01}} + x_{\widehat{23}})  
= \frac{1}{{\rm det} A_{a,1}}(c_{01}x_{01} + c_{23}x_{23} + c_{45}x_{45}),
$$             
where $c_{ij}$ are the minors of the matrix $A_{a,1}$ in the basis 
$x_{01},x_{23},x_{45}$. Now
$$
Q(f_a^{-1}(P_{\lambda_{a,1}}R)) 
= \frac{1}{{\rm det} A_{a,1}} x_{2233}(c_{01}x_{01} + c_{23}x_{23} + c_{45}x_{45})
= \frac{c_{01}}{{\rm det} A_{a,1}}x_{012233},
$$
and hence 
$$
\delta(a) 
= \delta\nu_{\lambda_{a,1}}(Q\otimes R) 
= \frac{c_{01}}{{\rm det} A_{a,1}}P_{\lambda_{a,1}}x_{012233}
=  \frac{c_{01}}{{\rm det} A_{a,1}}x_{001122334455}  =
$$


$$
\ = \ a^2-a + \frac{a^4-a^3+2a-1}{(a^3-3a+1)h^2 +a^2+a+1}.          
$$                   

Since $h$ is transcendental, and $a^3-3a+1$ has no rational roots, 
the above expression never vanishes for rational $a$.  
This provides infinitely many $\lambda_{a,1}$ 
with non-zero infinitesimal invariant, and hence
(by Lemma \ref{dnu-non-zero}) -- infinitely many 
non-torsion elements in ${\rm Griff}(X)$.
This yields our main result about the Griffiths group of 
the 5-fold quartic double solid $X = X_4$:

\begin{theorem} \label{thm: infinitely generated case 1}
For the general $X = X_4 \subseteq \PP^6(1^6;2)$ the Griffiths group 
${\rm Griff}^3_{\QQ}(X)$ is infinitely generated as a vector space over 
the rationals $\QQ$.                  
\end{theorem}                   
                 
\begin{proof}
It is sufficient to see that there exists an infinite 
sequence of integers $a_1,a_2,...$ such that 
$\delta(a_i) = \delta\nu_{\lambda_{a_i,1}}(Q\otimes R)$
are linearly independent over $\QQ$. 
For this we rewrite 
$$
\delta(a) =  {p_a} + \frac{q_a}{r_a t + s_a} 
$$
where 
$t = h^2$ 
and 
$p_a = a^2-a$, 
$q_a = a^4-a^3+2a-1$,  
$r_a = a^3-3a+1$,
$s_a = a^2+a+1$.
When the argument $a$ takes 
the values $i = 1,2,3,...$ then the real numbers 
$\delta(1),\delta(2),\delta(3),...$
generate an infinite dimensional vector space over $\QQ$.
The rest of the argument repeats the proof of Theorem 4.2 
in \cite{AC}. 
\end{proof}


\bigskip


\section{The Griffiths group of $X_{2.3} \subseteq \PP^7$}\label{x23}

\subsection{The Fano-Calabi-Yau fivefold $X_{2.3} \subseteq \PP^7$} 
Let 
$$
X = X_{2.3} = (q(x) = f(x) = 0)
$$ 
be a smooth complete 
intersection of a quadric, $q(x) = 0$, and a cubic, $f(x) = 0$, 
in $\PP^7(x) = \PP^7(x_0:...:x_7)$. 

Since the 5-fold $X = X_{2.3}$ is a complete intersection,
 by the Lefschetz hyperplane section theorem it follows that
all the primitive cohomology groups $H_o^{p,q}(X)$ 
for $p+q < 5 = {\rm dim}\ X$ are zero.  
By \S\ref{rxy-2} below, the middle Hodge numbers of $X$ are 

\begin{equation}\label{cohomology-of-x23}
h^{5,0}(X) = h^{0,5}(X) = 0, \ 
h^{4,1}(X) = h^{1,4}(X) = 1, \ 
h^{3,2}(X) = h^{2,3}(X) = 83. 
\end{equation}

Therefore 
the complete intersection of a quadric and a cubic 
$X_{2.3} \subseteq \PP^7$ is a FCY 5-fold.

\medskip

In the remainder of this section we shall verify that the Griffiths group 
${\rm Griff}_{\QQ}(X_{2.3})$ is infinitely generated as 
a vector space over the rationals.




\medskip

\subsection{Deformations of triples $(X,Y,\lambda)$ 
            and Noether-Lefschetz loci}\label{xyl} 
Let $X \subseteq \PP^N$ be a FCY manifold of dimension $2n+1$. 
Suppose that $X$ is an ample divisor 
in a Fano $(2n+2)$-fold $Y$. As in \S\ref{lyx}, we will assume that 
${\rm Pic}\ Y = \ZZ H$, $-K_Y = rH$ for some integer $r = r(X) \ge 2$, 
and $X \in |\cO_Y(d)|$ for some positive integer $d < r$. 
Denote by $\cX$ and $\cY$ the deformation spaces 
of $X$ and $Y$, and let $\cG$ be the incidence 
\begin{equation}\label{projections2}
\cX \stackrel{p}\longleftarrow \cG = \{ (X,Y): X\subseteq Y \} \stackrel{q}\longrightarrow \cY
\end{equation}
with its two natural projections $p$ and $q$. 

%
%
%
%
%

\medskip

Let $\lambda \in H^{n+1,n+1}_o(Y,\ZZ) \subseteq H^{2n}(Y)$ be the class 
of a primitive integer $(n+1)$-cycle $Z_\lambda$ on $Y$. 
The class $\lambda$ determines locally around $(Y,X) = (Y_o,X_o)$ 
a family $\cF_{\lambda} \subseteq \cG$  
defined by all local deformations $(X_t,Y_t)$ of $(X,Y)$ 
inside $\cG$ for which the class $\lambda \in H^{2n+2}(Y_t) = H^{2n+2}(Y)$ 
remains of type $(n+1,n+1)$. 

For fixed $X$, let $\cY_X = p^{-1}(X)$ be the family of all $(X,Y)$ 
such that $Y$ contains $X$, and suppose that $H^1(T_Y(-X)) = 0$.
Then the tangent space $T\cY_X$ at $(X,Y)$ is naturally identified 
with $H^1(T_Y(-X))$, see e.g. \cite{Tyu} or \cite{B-MS}. 
Now, by exchanging the places of $X$ and $Y$ from \S\ref{lyx}, 
and considering instead of inclusion $\lambda \subseteq Y \subseteq X$
the restriction $Y \supseteq \lambda \mapsto \lambda \cap X \subseteq X$,
we get a diagram

\begin{picture}(136,40)


\put(25,30){\makebox(0,0){$H^0(T_Y|_X)$}}
\put(37,30){\vector(1,0){6}}
\put(57,30){\makebox(0,0){$H^1(T_Y(-X))$}}
\put(70,30){\vector(1,0){11}}
\put(90,30){\makebox(0,0){$H^1(T_Y)$}}
\put(99,30){\vector(1,0){6}}
\put(116,30){\makebox(0,0){$H^1(T_Y|_X)$}}

\put(90,10){\makebox(0,0){$H^{n+3}(\Omega^{n+1}_Y)$}}

\put(90,27){\vector(0,-1){13}}
\put(70,27){\vector(1,-1){13}}

\put(67,20){\makebox(0,0){{\scriptsize $\lambda^{n+1,n+1}\circ \rho$}}}
\put(98,20){\makebox(0,0){{\scriptsize $\lambda^{n+1,n+1}$}}}
\put(75,32){\makebox(0,0){{\scriptsize $\rho$}}}

\end{picture}

\noindent
in which the map, $\rho: H^1(T_Y(-X)) \rightarrow H^1(T_Y)$, is interpreted as 
the Kodaira-Spencer map for the family $\cY_X$. 


Let $\cF \subseteq \cY$ be the Noether-Lefschetz locus 
of all $Y$ for which $H^{n+1,n+1}_o(Y,\ZZ) \not= 0$.
In particular, $\cF_{\lambda}$ is a component of $\cF$.

As in \S\ref{lyx}, for fixed $X$ 
one defines the Noether-Lefschetz locus, 
$\cF(X)$, inside $\cY_X$ to be the set of all $Y \in \cY_X$ 
that belong to $\cF(X)$; and define
$\cF(X)_{\lambda} \subseteq \cF(X)$ to be the set of all 
$Y \in \cY_X$ that belong to $\cF_{\lambda}$.  

For the given triple, $(X,Y,\lambda)$, 
let $T\cF_\lambda$ be the tangent space to 
$\cF_\lambda \subseteq \cY \times \cX$ 
at $(X,Y) = (X_o,Y_o)$. 
The following is the analog of Proposition \ref{lyx-nl}, 
in which the inclusion $Y \subseteq X$ is replaced by $X \subseteq Y$.

\begin{proposition}\label{xyl-nl}
Let $(X,Y,\lambda)$ be as above, let 
$p_*: T\cF_{\lambda} \rightarrow T{\cX}$ 
be the map induced by the projection $p: \cF_{\lambda} \rightarrow \cX$,
and suppose that the composition
$$
\lambda^{n+1,n+1}\circ \rho : H^1(T_{Y}(-X)) \rightarrow H^{n+3}(\Omega^{n+1}_Y)
$$
is an isomorphism. 
Then 

\begin{enumerate}
\item
 The map $p_*: T\cF_\lambda|_{Y,X} \rightarrow T{\cX}|_{X} = H^1(T_X)$   
is also an isomorphism, and hence the family $\cF_{\lambda}$ 
is smooth of codimension $h^{n+3,n+1}(Y)$ in $\cG$ at $(X,Y)$ 
and the projection, $p: \cF_\lambda \rightarrow \cX$, is an isomorphism  
over a neighborhood of $X$. 

\item There are infinitely many 0-dimensional components of the Noether-Lefschetz locus 
$\cF(X) \subseteq \cY_X$ which together form a countable 
subseteq in $\cY_X$.
\end{enumerate}
\end{proposition}

\begin{remark}
For a proof of \ref{xyl-nl} -- see Proposition 2.4.1 and 
the proof of Proposition 1.2.3 in \S 2 of  \cite{B-MS}.   
\end{remark}

\medskip

\subsection{The infinitesimal invariant of a normal function associated 
to a deformation of a triple $(X,Y,\lambda)$}\label{xyld}

${}$

\medskip

 Let $(X,Y,\lambda)$ be a triple which satisfies the conditions 
of Proposition \ref{xyl-nl}, and suppose 
as in \S\ref{lyxd} that the Hodge conjecture holds 
for $Y$. 

Then by (1) of Proposition \ref{xyl-nl}
the map $p_*: T\cF_\lambda \rightarrow T_\cX$ is a local 
isomorphism, and we can proceed as in \S\ref{lyxd} to 
define a normal function $\nu_\lambda: \cX \rightarrow \cJ(\cX)$ 
and its infinitesimal invariant 
$\delta{\nu}_{\lambda} \in ({\rm Ker} \overline{\nabla})^\vee$,
where $\overline{\nabla}: \cH^{n+1,n}_{\cX} \otimes T_\cX \longrightarrow \cH^{n,n+1}_{\cX}$
is the variation of Hodge structure for the Fano-Calabi-Yau  $(2n+1)$-fold $X$. 

The first difference between this situation and the one considered in  \S\ref{lyx} -- \S\ref{lyxd}  
is that instead of regarding $n$-cycles $Z_\lambda$ on $2n$-folds 
$Y \subseteq X$ as $n$-cycles on $X$, we instead consider $(n+1)$-cycles 
$Z_\lambda$ on $(2n+2)$-folds $Y \supseteq X$ and then look at 
the Abel-Jacobi map for their restrictions 
$Z'_{\lambda} = Z_{\lambda} \cap X$, which are already 
$n$-cycles on $X$.  

The second, and perhaps more important difference for our purposes, is 
that instead of varying $Y$ inside $X$,  $Y$ 
varies as a submanifold of $X$. In the former case, the way to interpret the infinitesimal invariant 
$\delta{\nu}_{\lambda}$ in terms of the graded ring of $Y$ was known already by Voisin.\footnote{see also \cite{W}.}

However, as far as the authors are aware, in our latter case there is 
no such translation. 
Nevertheless, as we shall see in the next subsection, in the example we consider of
$X = X_{2.3} \subseteq \PP^7$  this obstacle 
can be overcome.  This is essentially due to the observation that the general cubic 
6-folds, $Y = Y_3$, containing the given $X$ are the same as 
the general hyperplane sections of a nodal 7-fold cubic 
$Z = Z_3 \subseteq \PP^8$ uniquely attached to $X$.
With this observation, one can rewrite the infinitesimal 
invariant $\delta{\nu}_\lambda$ in terms of the graded rings 
of the cubic 6-folds $Y$.

%
%
%
%

\medskip

\subsection{The graded ring of $X_{2.3}$ 
            and of cubic $6$-folds $Y_3$}\label{rxy-2} Let
$$
X = X_{2.3} = (q(x) = f(x) = 0)
$$
be a complete intersection of a quadric $q(x) = 0$ and a cubic $f(x) = 0$ 
in $PP^7(x) = \PP^7(x_0:...:x_7)$.
To understand  the groups, $H_o^{5-p,p}(X)$, we follow \cite{Na} 
and use the Cayley trick to 
represent the primitive cohomology groups 
$H_o^{5-p,p}(X) = H_o^p(\Omega^{5-p}_X)$
as components of the bigraded ring of a hypersurface.

Let $W = \PP_{\PP^7}(\cO(-2)\oplus \cO(-3))$,
and let 
$D_X \subseteq W$ 
be the hypersurface defined by 
$$
F(x;y,z) = yf(x) + zq(x) = 0.
$$

Introduce bidegrees of the variables $(x) = (x_0:...:x_7)$, $y$ and $z$
as follows:
$$
{\rm deg}\ y = (1,-3),\ \ {\rm deg}\ z = (1,-2),\ \ {\rm deg}\ x_i = (0,1),\ i = 0,...,7.
$$

In the bigraded polynomial ring 
$$
S(X) := \CC[x_0,....,x_7,y,z],
$$ 
denote by 
$F_{x_i} = \frac{\partial F}{\partial x_i}$, $i = 0,..,7$,
$F_y = \frac{\partial F}{\partial y} = f(x)$, and 
$F_z = \frac{\partial F}{\partial z} = q(x)$ the 
partial derivatives of $F = F(x;y,z) = F(x_0,...,x_7;y,z)$.
Let 
$$
J(X) \ = \ 
\big( \frac{\partial F}{\partial x_0}, ..., \frac{\partial F}{\partial x_7},
\frac{\partial F}{\partial y},\frac{\partial F}{\partial z} \big) 
\  = \  
\big< \frac{\partial F}{\partial x_0}, ..., \frac{\partial F}{\partial x_7},f(x),q(x) \big>
$$ 
be the Jacobian ideal of $F$, 
and let 
$$
R(X) = S(X)/J(X) = \bigoplus_{a,b} R_{a,b}(X)
$$
be the {\it Jacobian ring} of $X = X_{2.3}$ 
decomposed into bigraded parts $R_{a,b}(X)$. 
Then
$$
H_o^{5-p,p}(X) = R_{p,-3}(X), \ \mbox{ for } p = 0,...,5
$$
-- see \cite{Na}.
Since all Hodge numbers, $h^{p,q}$, with $p+q$ odd  come from primitive classes, we get

\begin{equation}\label{cohomology-of-x23-again}
h^{5,0}(X) = h^{0,5}(X) = 0, \ 
h^{4,1}(X) = h^{1,4}(X) = 1, \ 
h^{3,2}(X) = h^{2,3}(X) = 83. 
\end{equation}










\medskip

\subsection{Cycles on cubic 6-folds $Y \supseteq X$}\label{zly}

Let $X = X_{2.3} = (q(x) = f(x) = 0)$ 
be a general smooth complete intersection of a quadric 
and a cubic in $\PP^7$, and let $Y = Y_3 \subseteq \PP^7$ be a smooth cubic 
6-fold containing $X$. Let $\cX$ and $\cY$ be the deformation spaces of 
$X$ and $Y$, and let 

$$
\cX \stackrel{p}\longleftarrow 
\cG = \{ (X,Y) : X \subseteq Y \} 
\stackrel{q}\longrightarrow \cY
$$ 

\medskip

\noindent
be the variety of pairs $(X,Y) \subseteq \cX \times \cY$
with $X \subseteq Y$ and its projections to $cX$ and $\cY$.

By \cite{Ste}, the Hodge conjecture holds for the cubic 
6-folds $Y$. 
Let $\lambda \in H^{3,3}_o(Y,\ZZ)$ be a Hodge class on $Y$, 
representing a primitive algebraic 3-cycle $Z_\lambda \subseteq Y \subseteq X$,
see \S\ref{xyld}. 
For the given triple $(\lambda,Y,X)$ the diagram from \S\ref{xyl}  
becomes

\begin{picture}(136,33)


\put(20,25){\makebox(0,0){$H^0(T_Y|_X)$}}
\put(32,25){\vector(1,0){6}}
\put(52,25){\makebox(0,0){$H^1(T_Y(-X))$}}
\put(65,25){\vector(1,0){11}}
\put(85,25){\makebox(0,0){$H^1(T_Y)$}}
\put(94,25){\vector(1,0){6}}
\put(111,25){\makebox(0,0){$H^1(T_Y|_X)$}}
\put(122,25){\vector(1,0){6}}
\put(127,25){\makebox(0,0){$...$}}

\put(85,5){\makebox(0,0){$H^4(\Omega^2_Y)$}}

\put(85,22){\vector(0,-1){13}}
\put(65,22){\vector(1,-1){13}}

\put(64,15){\makebox(0,0){{\scriptsize $\lambda^{3,3}\circ \rho$}}}
\put(90,15){\makebox(0,0){{\scriptsize $\lambda^{3,3}$}}}
\put(70,27){\makebox(0,0){{\scriptsize $\rho$}}}

\end{picture}

As in \S\ref{yx}, we rewrite this diagram in terms of the 
graded ring $R(Y)$. For this we first note the following 
identifications that can be obtained directly by using 
the adjoint and the tangent sequences for $X \subseteq Y \subseteq \PP^7$
and Bott vanishing: 

\begin{equation}
H^1(T_Y(-X)) \cong R_1(Y)\ , \ \ 
H^1(T_Y) \cong R_3(Y)\ , \ \
H^4(\Omega^2_Y) = H^{2,4}(Y) \cong R_7(Y).
\end{equation}


Next, as in \S\ref{lyx}, the composition 
$$
\lambda^{3,3}\circ \rho :
H^1(T_Y(-X)) \stackrel{\rho} \longrightarrow H^1(T_Y) 
\stackrel{\lambda^{3,3}} \longrightarrow H^4(\Omega^2_Y)
$$
becomes
$$
P_{\lambda} \circ e : 
R_1(Y) \stackrel{e} \longrightarrow R_3(Y) 
\stackrel{P_{\lambda}} \longrightarrow R_7(Y),
$$
where $P_\lambda:R(Y) \rightarrow R(Y)$ 
is the multiplication by the polynomial 
class $P\lambda \in R_4(Y)$ corresponding 
to $\lambda \in H^{3,3}_o(Y) \cong R_4(Y)$,
and $e: R(Y) \rightarrow R(Y)$ is the 
multiplication by the class of the quadric $q(x)$.




\medskip

\subsection{The v.H.s. for $X_{2.3}$ in terms of the bigraded ring $R(X)$}\label{vhsrx}
As in \S\ref{dry-1}, start from the variation of Hodge structure (v.H.s.) 
for $X = X_{2.3}$
$$
\overline{\nabla}: H^{3,2}(X) \otimes H^1(T_X) \rightarrow H^{2,3}(X).
$$


Since $X$ is a FCY manifold, the cup-product 
with the unique (modulo $\CC^*$) form $\omega^{4,1}$ on $X$ 
defines an isomorphism
$$
H^1(T_X) = H^{-1,1}(X) \stackrel{\omega^{4,1}}\longrightarrow H^{3,2}(X).
$$
By \cite{Na},
$H^{3,2}(X) \cong R_{2,-3}(X)$, $H^{2,3}(X) = R_{3,-3}(X)$; 
and under these isomorphisms the v.H.s. $\overline{\nabla}$
translates to a map 
$$
\mu_X: R_{2,-3}(X) \otimes R_{2,-3}(X) \rightarrow R_{3,-3}(X)
$$

However in this case $\mu_X$ is not given by multiplication 
of polynomials (modulo the Jacobian ideal) as in \S\ref{yx}  - for example multiplication would be additive on
bidegrees.
Fortunately, by using the generic 1:1 correspondence between $X = X_{2.3}$ and nodal 
cubic 7-folds $Z = Z_3$ from \S\ref{xz} below, 
we are able to instead rewrite $\mu_X$ as multiplication $\mu_Z$ in the graded ring $R(Z)$.




\subsection{$X_{2.3}$ and nodal cubic sevenfolds}\label{xz}
Let us think of projective 7-space, $\PP^7$, as the hyperplane $(w = 0)$ 
in the projective 8-space, $\PP^8 = \PP^8(x;w) = \PP^8(x_0:...:x_7:w)$.
Let $Z = Z_3 \subseteq \PP^8$ be a general nodal cubic sevenfold, 
and let $o \in Z$ be the node of $Z$;  Without loss of generality 
we may assume that $p_o = (0:...:0:1)$. 
The rational projection 
$$
p_o: \PP^8(x;w) \DashedArrow[densely dotted    ]
 \PP^7(x), 
$$
from $o$ sends the cubic, $Z$, birationally to $\PP^7(x)$.
Under a slight abuse of notation, we denote this birational map by 
$$
p_o: Z \DashedArrow[densely dotted    ]
 \PP^7.
$$ 
as well.
Since $Z$ has a node at the point $o = (0;1)$, 
then in the same coordinates, $(x;w)$, the equation 
of the cubic, $Z \subseteq \PP^8(x;w)$, can be written as
$$
f(z;w) = f(x) + q(x)w = 0
$$
where $f(x)$ is a cubic form of $(x) = (x_0:...:x_7)$ 
and $q(x)$ is a non-degenerate quadratic form in $(x)$. 

Let $\sigma: \tilde{Z} \rightarrow Z$ be the blowup 
of $Z$ at $o$, and let $E = \sigma^{-1}(o) \subseteq \tilde{Z}$ 
be the exceptional divisor over $o$; the divisor $E$ is 
isomorphic to a smooth 6-fold quadric identified with 
the base $Q = (q(x) = 0)$ of the projective 
tangent cone to $Z$ at $o$. 

The family of lines $L \subseteq Z$ that pass through the point $o$ 
sweep out a cone $R_o$ with vertex $o$ and a base given by the 5-fold 
$$
X = \{ x : q(x) = f(x) = 0 \} \subseteq \PP^7(x).
$$
For the general choice of the nodal cubic, $Z$, the 5-fold, $X = X_{2.3}$, 
is a general smooth complete intersection of a quadric and a cubic in $\PP^7$.

\begin{lemma}\label{x-z}
In the above notation, let 
$$
Z = Z_3 = (f(x) + q(x)w = 0) \subseteq \PP^8(x;z)
$$ 
be the general nodal cubic 7-fold with node $o = (0;1)$. 
Then 

\begin{enumerate}
\item The birational map 
$p_o: Z \DashedArrow[densely dotted    ]
 \PP^7(x)$ 
induced by the rational projection $\PP^8(x;w) \DashedArrow[densely dotted    ] \PP^7(x)$
decomposes as in the diagram below

\begin{picture}(90,60)

\put(35,20){\makebox(0,0){o}}
\put(40,20){\makebox(0,0){$\in$}}
\put(45,20){\makebox(0,0){$Z_3$}}
\put(75,20){\makebox(0,0){$\PP^7$}}
\put(81,20){\makebox(0,0){$\supseteq$}}
\put(89,20){\makebox(0,0){$X_{2.3}$}}

\put(35,50){\makebox(0,0){$E$}}
\put(40,50){\makebox(0,0){$\subset$}}
\put(45,50){\makebox(0,0){$\tilde{Z}$}}
\put(53,50){\makebox(0,0){$\supseteq$}}
\put(60,50){\makebox(0,0){$\tilde{R}_o$}}

\put(35,45){\vector(0,-1){20}}
\put(45,45){\vector(0,-1){20}}
\put(48,45){\vector(1,-1){20}}
\put(64,45){\vector(1,-1){20}}
\put(49,20){\vector(1,0){20}}

\put(60,17){\makebox(0,0){$p_o$}}
\put(43,35){\makebox(0,0){$\sigma$}}
\put(64,35){\makebox(0,0){$\tilde{p}_o$}}

%
\put(45,13){\makebox(0,0){$\cup$}}
\put(45,06){\makebox(0,0){$R_o$}}
\put(57,46){\vector(-1,-4){09}}
%

\end{picture}

\noindent
where $\tilde{R}_o \subseteq \tilde{Z}$ the proper preimage of the 
cone $R_o \subseteq Z$, and $\tilde{p}_o: \tilde{Z} \rightarrow \PP^7 = \PP^7(x)$  
is a birational morphism contracting $\tilde{R}_o$ to the complete 
intersection 
$$
X_{2.3} = (q(x) = f(x) = 0)
$$
of the quadric $(q(x) = 0)$ and the cubic $(f(x) = 0)$ in $\PP^7$. 

\item
The birational morphism $\tilde{p}_o: \tilde{Z} \rightarrow \PP^7$ 
from (i) coincides with the blow-up of $X = X_{2.3}$ in $\PP^7$. 
Moreover the inverse to (i) takes place: 

If $X_{2.3} = (q(x) = f(x) = 0)$ be a general complete intersection 
of a quadric $(q(x) = 0)$ and a cubic $(f(x) = 0)$ in $\PP^7 = \PP^7(x)$, 
then the blowup of $\PP^7$ at $X$ is the same as the blowup $\tilde{Z}$ 
of the nodal cubic 7-fold $Z = Z_3 = (f(x) + q(x)w = 0)$. 

\end{enumerate}
\end{lemma}

\begin{proof} 
The proof is straightforward, we leave the details as an exercise to the reader.
\end{proof}

\medskip

\subsection{The infinitesimal invariant of $(X,Y,\lambda)$ 
            by equivalences of graded rings}\label{xyldr}

Below, we rewrite the infinitesimal invariant $\delta{\nu}_\lambda$ 
for a triple $(X,Y,\lambda) = (X_{2.3},Y_3,\lambda)$ (as in \S\ref{xyl}), 
by using the equivalence of graded rings of $X_{2.3}$ and the nodal cubic 
7-fold $Z = Z_3$ corresponding to $X$.  However, in order to proceed, we first require the following lemma:




\begin{lemma}
Let $X = X_{2.3} \subseteq \PP^7$ be the general complete intersection
of a quadric and a cubic, and let $Z = Z_3 \subseteq \PP^8$ be the nodal 
cubic 7-fold corresponding to $X$ by Lemma \ref{x-z}. 
Then a 6-fold cubic, $Y$, containing $X$ and not containing 
the quadric, $q(x) = 0$, as a component
can be identified with a hyperplane sections $Y$ of $Z$ 
which do not pass through the node $o$ of $Z$. In particular, 
the generic cubic 6-fold containing the 5-fold $X = X_{2.3}$ 
is the same as 
the generic hyperplane section of its corresponding cubic 7-fold 
$Z = Z_3$.
\end{lemma}

\begin{proof}
In the notation of \S\ref{rxy-2}, 
let 
 $X = X_{2.3} = (q(x) = f(x) = 0) \subseteq \PP^7(x)$, 
and let 
$Z = Z_3 = (f(x) + wq(x) = 0) \subseteq \PP^8(x;w)$ 
be the nodal cubic corresponding to $X$ by Lemma \ref{x-z}. 
A cubic 6-fold $Y \subseteq \PP^7(x)$ containing $X$ and not containing 
the quadric $(q(x) = 0)$ as a component has the equation 
$$
Y = (f(x) + l(x)q(x) = 0)
$$
where $l(x)$ is a non-zero linear form of $(x) = (x_0:...:x_7)$.

From the equation $F(x;w) = f(x) + wq(x)$ of the nodal 
cubic 7-fold $Z = Z_3$ corresponding to $X$,  
we see that $Y$ is the same as the linear section 
$$
Y = Z \cap (l(x) - w = 0)
$$
of the cubic $Z$. Since in the linear form, $l(x) - w$, the coefficient at $w$ 
is non-zero, it follows that the node, $o = (0;1)$, of $Z$ does not lie 
on the linear section $Y \subseteq Z$.  Furthermore, since $l(x)$ is a general linear form on $(x)$, 
the form, $l(x) - w$, defines the general linear 
section of $Z$ that does not pas through its node, $o$.
\end{proof}  

\medskip

\subsubsection{\bf The bigraded Jacobian ring $R(X)$}
Let $X = (q(x) = f(x) = 0)$,
where $2q(x) = x_0^2+...+x_7^2$, $3f(x) = x_0^3+...+x_7^3$.
Then 
$R(X) = S(X)/J(X)$ is the same as $\CC[x_0,...,x_7,y,z]$ 
modulo the relations 
$$
x_0^3+...+x_7^3 = x_0^2+...+x_7^2 = 0 \ \mbox{ and } \ yx_i^2 = zx_i, i = 0,...,7.
$$ 
A simple combinatorial check yields:

\begin{lemma} 
Let $R(X)$ be the graded ring of $X = X_{2.3}$ with equations chosen
as above. Then: 
 
 \begin{enumerate}
 
\item  $H^{4,1}(X) = R_{1,-3}(X) = \CC y$.

\item $H^{3,2}(X) = R_{2,-3}(X)$ is generated over $\CC$ 
by the following 92 monomials:

\medskip

\begin{itemize}

\item the 8 monomials:  $z^2x_i, i = 0,...,7$;


\item the 28 monomials:  $yzx_ix_j, 0 \le i < j \le 7$; 


\item and the 56 monomials: $y^2x_ix_jx_k, 0 \le i < j < k \le 7$. 

\end{itemize}

\end{enumerate}

\end{lemma}


%
%
%

There are $9$ independent relations between the 92 monomials from (2). 
These relationships can be found, using the following identities in $R(X)$,
\[
yx_i^2 \mapsto zx_i, \ i = 0,...,7,
\]
which yield the following identities in the graded component $R_{2,-3}(X)$, 	
\[
y^2x_i^3 =  yzx_i^2 = z^2x_i \text{ and } y^2x_i^2x_j = yzx_ix_j.
\]

Now, using the basis from (2), we get 1 relation between $z^2x_i = y^2x_i^3$
coming from $3f(x) = x_0^3+...+x_7^3 = 0$,
and  8 relations between $y^2x_i^2x_i$
coming from $2q(x)x_i = (x_0^2+...+x_7^2)x_i = 0$.
One can easily verify that all the relations 
between the generators from (2) are generated by the above 
9 relations.  In particular, ${\rm dim}\ H^{3,2}(X) = 92 - 9 = 83$.











\medskip

\subsubsection{\bf The graded ring, $R(Z)$, of the nodal cubic 7-fold, $Z$, corresponding to $X$}

For the above choice of $X \subseteq \PP^7(x)$ 
the nodal cubic, $Z \subseteq \PP^7(x;w)$, corresponding to $X$ is 
$$
Z = (6f(x;w) = 2f(x) - 3wq(x) = 2(x_0^3 + ... + x_7^3) - 3w(x_0^2 + ... + x^2) = 0)
$$ 
In the graded ring, $S(Z) = \CC[x;w] = \CC[x_0,..,x_7,w]$, 
the Jacobian ideal, $J(Z)$, of $Z$ is generated by 
the relations
$$
f_{x_i} = x_i^2 - wx_i = 0 , i = 0,..,7
\ \mbox{ and } \  
f_w = x_0^2 + ... + x_7^2 = 0.
$$

Let 
$$
R(Z) = S(Z)/J(Z) = \bigoplus_{d\ge 0} R_d
$$ 
be the graded Jacobian ring of $Z$.
Then the component $R_3(Z)$ is generated by:

\begin{itemize} 
\item the 56 monomials: $x_ix_jx_k, 0 \le i < j < k \le 7$; 

\item  the 28 monomials: $wx_ix_j = x_i^2x_j = x_ix_j^2, 0 \le i < j \le 7$; 

\item the 8 monomials: $w^2x_i = wx_i^2 = x_i^3, i = 0,...,7$;

\item and the monomial: $w^3$. 
\end{itemize}

\begin{lemma}\label{rx-rz}
Let $X$ and $Z$ be as above. 
Then the $\CC$-linear map, $S(X) = \CC[x;y,z] \rightarrow S(Z) = \CC[x;w]$, 
defined by 
$$
y \mapsto 1, \ \ z \mapsto w, \ \ x_i \mapsto x_i, i = 0,...,7
$$
factors through the Jacobian ideals $J(X)$ and $J(Z)$.
Let
$$
j: R(X) = S(X)/J(X) \rightarrow S(Z)/J(Z) = R(Z)
$$
be the induced map, 
$S^o(Z) = \CC[x;w]^o$ be the set of all $g(x;z)$ that vanish
at the node $o = (0;1)$ of $Z$, 
$J^o(Z) = J(Z) \cap S^o(Z)$, 
and 
$$
R^o(Z) = S^o(Z)/J^o(Z).
$$
The map $j: R(X) \rightarrow R(Z)$ restricts to isomorphisms:
$$
j: R_{2,-3}(X) \stackrel{\sim}\longrightarrow R_3^o(Z)
$$ 
and 
$$
j: R_{3,-3}(X) \stackrel{\sim}\longrightarrow R_6^o(Z).
$$ 
\end{lemma}

\begin{proof}
As a model, we use the 5-fold, $X_{2.3}$, defined above.  We will 
verify the isomorphism $j$ only for $R_{2,-3}(X)$.  
The computation for a general $X_{2.3}$ and for $R_{3,-3}(X)$ does not differ substantially.

By the preceding discussion, the component $R_{2,-3}(X)$ 
is generated by 
$y^2x_ix_jx_k$, $yzx_ix_j$ and $z^2x_i$ 
with relations $yx_i = zx_i^2$, $\Sigma\ x_i^2 = 0$, and
$\Sigma\ x_i^3 = 0$. 

On the one hand, 
the map $j$, as defined above, sends the generating monomials of $R_{2,-3}(X)$ 
to $x_ix_jx_k$, $wx_ix_j$ and $w^2x_i$ respectively.
This follows from the preceding discussion together with the fact that the $w^3$ generate $S_3(Y)$. 
Since the hyperplane, $S^o_3(Y) \subseteq S_3(Y)$, is defined by $w^3 = 0$,
this yields that $j$ sends $S_{-2,3}(X)$ surjectively to $S_3(Z)$.  

On the other hand, 
$j$ sends the generating relations $yx_i^2 - zx_i$, $\Sigma\ x_i^2$ 
and $\Sigma\ x_i^3$ for $J(X)$ to 
$x_i^2 - wx_i$, $\Sigma\ x_i^2$ and $\Sigma\ x_i^3$ respectively.  These relations generate $J(Z)$, see above.    
Notice that $\Sigma\ x_i^3 = \Sigma\ wx_i^2 = w(\Sigma\ x_i^2)$ 
belongs to $J(Z)$, ibid.
\end{proof}

\medskip

\subsubsection{\bf The infinitesimal invariant by equivalence of graded rings}

From the identifications in Lemma \ref{rx-rz}, 
the map,
$$
\mu_X: R_{2,-3}(X) \otimes R_{2,-3}(X) \rightarrow R_{3,-3}(X),
$$
from \S\ref{vhsrx}
transforms to 
$$
\mu_Z: R^o_3(Z) \otimes R^o_3(Z) \rightarrow R^o_6(Z).
$$ 
Now, as in \S\ref{yx}, by \cite{Voi1} (see also \cite{AC}) 
one has:

\begin{corollary}
The map,
 $$
\mu_Z: R^o_3(Z) \otimes R^o_3(Z) \rightarrow R^o_6(Z),
$$  
is induced by multiplication of monomials in the graded 
ring $S(Z) = \CC[x;w]$.
\end{corollary} 




Let $(\lambda,Y,Z)$ be a triple consisting of a  
nodal cubic 7-fold, 
$$Z = (f(x;w) = f(x) + q(x)w = 0) \subseteq \PP^8(x;w),$$ 
a hyperplane section, 
$$Y = (f(x) = 0),$$ 
of $Z$,  
and a class, $\lambda$, representing an algebraic 3-cycle 
$Z_{\lambda}$ on $Y$. By the previous corollary and 
\cite{Voi1} the infinitesimal invariant $\delta_{\lambda}\nu$ 
can be interpreted as a linear form on the kernel 
of the multiplication map, 
$$
\mu_Y: R_3(Y) \otimes R_3(Y) \rightarrow R_6(Y),
$$
as follows (see also \S 2 in \cite{AC}):

\begin{lemma}\label{voisin-2}
Let $e = q(x) = \partial (f(x)+wq(x))/{\partial w}|_{w=0}(mod. J(Y)) \in R_2(Y)$,
and let $P_{\lambda}$ be the element of $R_4(Y) \cong H_o^{3,3}(Y)$ 
corresponding to $\lambda^{3,3}$.  
If the multiplication map, $P_{\lambda}\circ e :R_1(Y) \rightarrow R_7(Y)$,
defined in \S\ref{zly} is an isomorphism, then for 
any element, $\sum_a F_a \otimes G_a \in {\rm Ker}\ \mu_Y \subseteq R^o_3(Y) \otimes R^o_3(Y)$,
one has
$$
\delta_{\lambda}\nu \ (\sum_a F_a \otimes G_a) 
= 
\sum_a P_{\lambda}.F_a.(P_{\lambda} \circ e)^{-1}(P_{\lambda}.G_a) \in R^8(Y) \cong \CC.
$$
\end{lemma}

%
%

\medskip

\subsection{Cycles on the Fermat cubic sixfold}
Let $Y$ be the Fermat cubic sixfold, i.e., $X \subseteq Y \subseteq \P^7$ is defined by the equation, 
$$
f(x) = x_0^3 + ... + x_7^3.
$$
Let 
$$
R(Y) = S(Y)/J(Y) = \CC[x_0,...,x_7]/(x_0^2,....,x_7^2)
= \bigoplus_{d\ge0} R_d(Y)
$$ 
be the graded Jacobian ring of $Y$. By \cite{Na}, the primitive cohomology satisfies
$$
H^{3,3}_o(Y) \cong R_4(Y).
$$
Following \cite{Shi} and \S 3 in \cite{AC}, 
we now describe the rational cohomology 
classes in $H^{3,3}_o(Y)$ and their corresponding elements from $R_4(Y)$. 

Let $\mu_3$ be the group of $4^{\op{th}}$ roots of unity, $\zeta_i$, and let 
$G = (\mu_3)^8/\Delta$, where $\Delta$ is the diagonal subgroup. 
Setting $\ZZ_3 = \ZZ/3\ZZ$, the character group 
$\hat{G}$ of $G$ is naturally embedded in $(\ZZ_3)^8$ as 
$$
\hat{G} \stackrel{\sim}\longrightarrow \{ \alpha = (a_0,...,a_7): a_0+...+a_7 = 0 \} 
\subseteq (\ZZ_3)^8;
$$
the character $\alpha \in \hat{G} = \op{Hom}(G,\CC^*)$ representing $(a_0,...,a_5)$ sends 
the element $[\zeta_0,...,\zeta_7] \in G = (\mu_3)^8/\Delta$, 
to $\alpha([\zeta_0,...,\zeta_7]) = \zeta_0^{a_0}...\zeta_7^{a_7}$.

Let $\hat{G}^* = \{ \alpha = (a_0,...,a_7) \in \hat{G}: a_i \not=0, i = 0,...,7 \}$.
For $\alpha = (a_0,...,a_7) \in \hat{G}^*$, define its norm,
$$
|\alpha| = \frac{<a_0> + ... + <a_7>}{3},
$$ 
where $<a_i>$ is the unique integer between $1$ and $2$ 
congruent to $a_i$ modulo $3$. 
The natural action,
$$
g = [\zeta_0,...,\zeta_7]: (x_0:...:x_7) \mapsto (\zeta_0x_0,...,\zeta_7x_7),
$$
of $G$ on $\PP^7$ restricts to an action of $G$ on the Fermat cubic $Y \subseteq \PP^7$. 
This induces a representation $g^*$  of $G$ on the primitive cohomology 
group $H^6_o(Y,\CC)$. 
Let 
$$
V_{\alpha} = \{ \lambda \in H^6_o(Y,\CC): g^*(\lambda) = \alpha(g)\lambda \}
$$  
be the eigenspaces of $g^*$ in $H^6_o(Y,\CC)$
defined by the characters $\alpha \in \hat{G}$.  
By \cite{Shi} we have the following:

\begin{enumerate}
\item The primitive cohomology satisfies the identity,
$$
H^{3,3}_o(Y,\QQ)\otimes_\QQ \CC = \bigoplus_{\alpha \in B} V_{\alpha}, 
$$  
where 
$
B = \{ \alpha \in \hat{G}^* : |\alpha| = |2.\alpha| = 4 \} 
$
and $2.\alpha =2.(a_0,...,a_7) = (2a_0,...,2a_7)$. 

\item  Let $C_o(Y)_\QQ$ denote the subspace of $H^6_o(Y,\QQ)$ spanned by 
classes of primitive algebraic 3-cycles on $Y$. We have the following identification: 
$$
C_o(Y) = C_o(Y)\otimes_\QQ \CC = H^{3,3}_o(Y,\QQ)\otimes_\QQ \CC.
$$
\end{enumerate}

\begin{remark}
The space, $C_o(Y)$, spanned by classes of primitive algebraic 2-cycles 
on a fourfold, $Y$, is a subspace of $H^{3,3}_o(Y,\QQ)\otimes_\QQ \CC$.  The 
coincidence (2) is the statement of the Hodge conjecture for 
the Fermat cubic sixfold, see \cite{Shi}. 
\end{remark}

\medskip

\subsection{The isomorphism $H^{3,3}_o(Y) \rightarrow R_4(Y)$ in coordinates}
Notice that in $R(Y) = S(Y)/J(Y) = \CC[x_0,...,x_7]/(x_0^2,...,x_5^2)$ 
for a monomial, 
$x_0^{b_0}...x_5^{b_5}$, representing a non-zero class 
modulo $J(Y)$, each of the coordinates, $x_i$, has degree at most $1$. 
Therefore, $R_4(Y)$ is generated by the 70 monomials,
$$
x_{i_1}x_{i_2}x_{i_3}x_{i_4},\ 0 \le i_1 < i_2 < i_3 < i_4 \le 7,
$$
regarded as classes modulo $J(Y) = (x_0^2,...,x_7^2)$.

By definition, $\alpha = (a_0,...,a_7) \in B$ 
if and only if $|\alpha| = |2.\alpha| = 4$. Since $B \subseteq \hat{G}^*$,
the coordinates, $a_i$, take values $k = 1$ or $2$. 

For an element, $\alpha = (a_0,...,a_7) \in \hat{G}^*$,
let  
$$
d_k(\alpha) = \# \{ i : a_i = k \}, 
$$
be the number of occurrences of the number $k \in \{ 1,2 \}$ 
among the coordinates $a_i$ of $\alpha$.
We call $\alpha$ an element of type $(2^{p}1^{q})$ 
if $d_2(\alpha) = p$ and $d_1(\alpha) = q$. 
As in \cite{AC}, the isomorphism 
$$
j: H_o^{3,3}(Y) \rightarrow R_3(Y)
$$
is given by:
$$
j:\alpha = (a_0,...,a_7) \longmapsto x_o^{a_0-1}x_1^{a_1-1}...x_5^{a_5-1}.
$$
This directly implies that all elements $\alpha \in B$ are of type 
$(2^41^4)$, sent by $j$ to the 70 monomials $x_{i_1}x_{i_2}x_{i_3}x_{i_4}$ 
as above.

\medskip

\subsection{Infinite generation of the Griffiths group of $X_{2.3} \subseteq \PP^7$.}
As above, we continue to have $X \subseteq Y \subseteq \PP^7(x) = \PP^7(x_0,...,x_7)$, where $Y$ is the 
Fermat cubic 6-fold.
To simplify notation, for $0 \le i \le j \le ... \le k$ we write
$$
x_{ij...k} = x_ix_j...x_k,
$$ 
for the monomial as well its class $x_{ij...k}$ in $R(Y)$.
For example, $x_{22122111} = x_0x_1x_3x_4$.
We call two monomials $x_{ij...k}$ and $x_{i'j'...k'}$ 
{\it dual}, and write $x_{i'j'...k'} = x_{\widehat{ij...k}}$
if \  $x_{ij...k} . x_{i'j'...k'} = x_{01234567}$.
Let 
$$
P_{\lambda} = P_{\lambda_{a,b}} = a(x_{0123} + x_{4567}) + b(x_{0124} +x_{3567}), 
$$ 
and
$$
e = x_{01} + x_{23} + x_{45} + x_{67} + hx_{35},
$$
where $h$ is a transcendental number.
Then 
$$
P_{\lambda_{a,b}}.e 
= a(x_{\widehat{01}}+ x_{\widehat{23}} + x_{\widehat{45}} + ax_{\widehat{67}})
+ b(x_{\widehat{24}} + bx_{\widehat{35}}) + hx_{\widehat{67}}.
$$

We will now verify that for generic $a$ and $b$ the linear map 
$$
P_{\lambda_{a,b}}.e: R_1(Y) \rightarrow R_7(Y)
$$ 
is an isomorphism, see Proposition \ref{voisin-2}.
In the bases $x_{i}$ and $x_{\widehat{j}}$  of $R_1(Y)$ and $R_7(Y)$ 
respectively, $P_{\lambda_{a,b}}.e$ acts as follows:

\[
\begin{array}{cccc}
{} & x_{0} \mapsto ax_{\widehat{1}}, & x_{1} \mapsto ax_{\widehat{0}}, & {} \\
  x_{2} \mapsto ax_{\widehat{3}} + bx_{\widehat{4}}, 
& x_{3} \mapsto ax_{\widehat{2}} + bx_{\widehat{5}},
& x_{4} \mapsto bx_{\widehat{2}} + ax_{\widehat{5}},
& x_{5} \mapsto bx_{\widehat{3}} + ax_{\widehat{4}} \\
{} & x_{6} \mapsto (a+h)x_{\widehat{7}}, & x_{7} \mapsto (a+h)x_{\widehat{6}}. & {}
\end{array}
\]

\medskip

Therefore, in the bases $x_1,...,x_7$ and $x_{\widehat{1}}, ... , x_{\widehat{7}}$,
the matrix $M_{a,b}$ of $P_{\lambda_{a,b}}.e$ is 
$$
\left( \begin{array}{cc}
0  & a \\ 
a & 0
\end{array} \right)
\oplus 
{\left( \begin{array}{cccc}
0 & a & b & 0\\ 
a & 0 & 0 & b\\ 
b & 0 & 0 & a\\ 
0 & b & a & 0
\end{array} \right)}
\oplus
{\left( \begin{array}{cc}
0   & a+h \\ 
a+h & 0
\end{array} \right)}.
$$
Therefore, if ${\rm det} M_{a,b} = a^2(a-b)^2(a+b)^2(a+h)^2 \not=0$,
then $P_{\lambda_{a,b}}.e: R_1(Y) \rightarrow R_7(Y)$ is an isomorphism 
and we can apply Voisin's formula from Lemma \ref{voisin-2} 
to compute the infinitesimal invariant   
$\delta_{\lambda_{a,b}}$. 

As in \cite{AC}, we can restrict to the case $b = 1$
and compute $\delta_{\lambda_{a,1}}$
at a well chosen element $Q \otimes R \in {\rm Ker}(\mu_Y)$.
In this particular case $P_{a,1}.e$ is an isomorphism if $a(a-1)(a+1) \not= 0$.

Let $Q = x_{2233}$ and $R = x_{0123}$. Then $Q \otimes R \in {\rm Ker}(\mu_Y)$, 
and by Lemma \ref{voisin-2} the infinitesimal invariant  
$$
\delta_{\lambda_{a,b}}(Q \otimes R) = P_{\lambda_{a,1}}Q(f_a^{-1}(P_{\lambda_{a,1}}R)),
$$
where $f_a^{-1}: R_7(Y) \rightarrow R_1(Y)$ is the inverse to the isomorphism 
$f_a = P_{\lambda_{a,1}}$. 
Now the computation of the infinitesimal invariant is straightforward: 
$$
P_{\lambda_{a,1}}R = bx_{\widetilde{6}} 
\ \mbox { and } \ 
f_a^{-1}(P_{\lambda_{a,1}}R) = f_a^{-1}(bx_{\widetilde{6}})  
= \frac{b}{h+a} x_{7},
$$      
which, multiplied by $Q = x_{456}$, gives $\frac{b}{h+a} x_{4567}$.  Thus we have,
$$
\delta\nu_{\lambda_{a,1}}(Q\otimes R) 
= \frac{b}{h+a}P_{\lambda_{a,1}}x_{4567}
= \frac{b}{h+a}x_{01234567}.
$$

We are now ready to state the main result for $X_{2.3}$:            
                   
\begin{theorem} \label{thm: infinitely generated case 2}
For the general $X = X_{2.3} \subseteq \PP^7$ the Griffiths group 
${\rm Griff}^3_{\QQ}(X)$ is infinitely generated as a vector space over 
the rationals $\QQ$.                  
\end{theorem}                   
                 
\begin{proof}
Again, we need to check that for an infinite number of choices
$a_1,a_2,a_3,...$ of the integer parameter one gets a sequence 
of numbers  $\delta(a_i) = \frac{b}{h+a_i}$ 
that are linearly independent over $\QQ$. 
This is straightforward - see the proof of Theorem 4.2 in \cite{AC} (Albano and Collino prove this for the same coefficient 
function, $\frac{b}{h+a}$, obtained in the slightly different situation of a cubic $7$-fold).
\end{proof}                 

\medskip

\section{Griffiths groups for subcategories and quotients} \label{sec: admissible}
Let $\dbcoh{X}$ denote the bounded derived category of coherent sheaves on $X$ as a triangulated category, i.e., the category of complexes of coherent sheaves where maps which induce isomorphisms on cohomology are inverted.  For a thick triangulated subcategory, $\mathcal N \in \dbcoh{X}$, we define  $\op{K}^{\op{alg}}_0(\mathcal N)$ to be the subgroup of $ \op{K}^{\op{alg}}_0(X)$ generated by elements of $\mathcal N$.  Here, thick means that the triangulated category is closed under taking summands.  We can also consider the Verdier quotient $\dbcoh{X} / \mathcal N$  \cite{Ve}.

We can define a type of Griffiths groups for triangulated subcategories and quotients as follows:
\begin{definition}
Let $\mathcal N$ be a thick triangulated subcategory of $\dbcoh{X}$.  
The \emph{total rational Griffiths group} of $\mathcal N$ is 
\[
{\rm Griff}_{\QQ}(\mathcal N) := \op{ker} (c \circ \op{ch}_{sst} \circ i) \otimes_{\ZZ}\QQ
\]
where $i$ is the inclusion, $\op{ch}_{sst}$ is the chern character map landing in the Chow ring and $c$ is the cycle class map to cohomology.
The \emph{total rational Griffiths group} of the Verdier quotient , $\dbcoh{X} / \mathcal N$, is the quotient
\[
{\rm Griff}_{\QQ}(X)/{\rm Griff}_{\QQ}(\mathcal N). 
\]
 \end{definition}
 
A special case of Verdier quotients is when the subcategory $\mathcal N$ admits a left or right adjoint.  In this case, we will use the notation $\mathcal A$ and we will see below that the total rational Griffiths group of $\mathcal A$ has some very nice properties.
\begin{definition}
 Let $i: \mathcal A \to \mathcal T$ be the inclusion of a full triangulated subcategory of $\mathcal T$. The subcategory, $\mathcal A$, is called \emph{right admissible} if $i$ has a right adjoint $i^!$ and \emph{left admissible} if it has a left adjoint, $i^*$. A full triangulated subcategory is called \emph{admissible} if it is both right and left admissible. 
\end{definition}

Let $\mathcal T$ be a triangulated category and $\mathcal I$ a full subcategory. Recall that the left orthogonal, $\leftexp{\perp}{\mathcal I}$, is the full subcategory of $\mathcal T$ consisting of all objects, $T \in \mathcal T$, with $\op{Hom}_{\mathcal T}(T,I) = 0$ for any $I \in \mathcal I$. The right orthogonal, $\mathcal I^{\perp}$, is defined similarly.

A closely related notion to an admissible subcategory is that of a semi-orthogonal decomposition.
\begin{definition}\label{def:SO}
 A \emph{semi-orthogonal decomposition} of a triangulated category, $\mathcal T$, is a sequence of full triangulated subcategories, $\mathcal A_1, \dots ,\mathcal A_m$, in $\mathcal T$ such that $\mathcal A_i \subseteq \mathcal A_j^{\perp}$ for $i<j$ and, for every object $T \in \mathcal T$, there exists a diagram:
\begin{equation} \label{eq: semi}
\begin{tikzpicture}[description/.style={fill=white,inner sep=2pt}]
\matrix (m) [matrix of math nodes, row sep=1em, column sep=1.5em, text height=1.5ex, text depth=0.25ex]
{  0 & & T_{m-1} & \cdots & T_2 & & T_1 & & T   \\
   & & & & & & & &  \\
   & A_m & & & & A_2 & & A_1 & \\ };
\path[->,font=\scriptsize]
 (m-1-1) edge (m-1-3) 
 (m-1-3) edge (m-1-4)
 (m-1-4) edge (m-1-5)
 (m-1-5) edge (m-1-7)
 (m-1-7) edge (m-1-9)

 (m-1-9) edge (m-3-8)
 (m-1-7) edge (m-3-6)
 (m-1-3) edge (m-3-2)

 (m-3-8) edge node[sloped] {$ | $} (m-1-7)
 (m-3-6) edge node[sloped] {$ | $} (m-1-5)
 (m-3-2) edge node[sloped] {$ | $} (m-1-1)
;
\end{tikzpicture}
\end{equation}
 where all triangles are distinguished and $A_k \in \mathcal A_k$. We shall denote a semi-orthogonal decomposition by $\langle \mathcal A_1, \ldots, \mathcal A_m \rangle$.  
\end{definition}

\begin{definition}
 An object, $E$, in a $k$-linear triangulated category, $\mathcal T$, is called \textbf{exceptional} if, 
\[
\op{Hom}_{\mathcal T}(E,E[i])  = 
\begin{cases}
  k & \text{ if } i=0  \\
0 & \text{ if } i \neq 0. \\
 \end{cases}
\]
\end{definition}

 When $E$ is either an exceptional object of a $k$-linear triangulated category, then the inclusion of $\langle E \rangle$, the smallest triangulated category generated by $E$, has right adjoint $\op{Hom}(E, -) \otimes_k E$ and left adjoint $\op{Hom}(-, E)^{\vee} \otimes_k E$.    Hence, $\langle E \rangle$ is admissible.  Moreover, this category is equivalent to the derived category of vector spaces over $k$.  When this category appears in a semi-orthogonal decomposition, we follow conventions by just writing $E$ instead of $\langle E \rangle$ in the notation.

The following is Lemma 2.4 of \cite{Kuz09a}.
\begin{lemma} \label{lem: kuz}
 Let
 \[
 \mathcal T = \langle \mathcal A_1, ... , \mathcal A_m \rangle
 \]
be a semi-orthogonal decomposition.  For any $T \in \mathcal T$, the diagram \eqref{eq: semi} is unique and functorial.
\end{lemma}
This allows us to define the following functors.
\begin{definition}
The $k^{\op{th}}$-projection functor
\begin{align*}
\alpha_i: & \mathcal T  \to  \mathcal T\\
& T  \mapsto A_i
\end{align*}
and sends morphisms to those induced by Lemma~\ref{lem: kuz}.
\end{definition}
\begin{definition}
The $i^{\op{th}}$-truncation functor 
\begin{align*}
\tau_i: & \mathcal T \to  \mathcal T_i \\
& T  \mapsto  T_i
\end{align*}
and sends morphisms to those induced by Lemma~\ref{lem: kuz} between diagrams.
\end{definition}


\begin{definition} \label{def: Serre functor triangulated}
 Let $\mathcal T$ be a $k$-linear triangulated category with finite dimensional morphism spaces.  An autoequivalence, $S$, is called a \emph{Serre functor} for $\mathcal T$ if for any two objects, $T_1,T_2 \in \mathcal T$, there is a natural isomorphism,
 \begin{displaymath}
  \op{Hom}_{\mathcal T}(T_2, S(T_1)) \cong \op{Hom}_k(\op{Hom}_{\mathcal T}(T_1,T_2),k).
 \end{displaymath}
 \end{definition}

\begin{definition}
A triangulated category is called Calabi-Yau of dimension $n$ if there is an isomorphism of functors,
\[
S \cong [n],
\]
for some $n$.
\end{definition}

The proofs of the following lemmas can be found in \cite{BK}:

\begin{lemma}
Let $\mathcal A$ be a full triangulated subcategory of a triangulated category $\mathcal T$ possessing a Serre functor.  Then the following are equivalent:
\renewcommand{\labelenumi}{\emph{\roman{enumi})}}
\begin{enumerate}
 \item $\mathcal A$ 	is left admissible
 \item $\mathcal A$ is right admissible
 \item $\mathcal A$ is admissible
\end{enumerate}
\end{lemma}

\begin{lemma} \label{lem: SOD admissible}
If $\langle \mathcal A_1, \ldots, \mathcal A_m \rangle$ is a semi-orthogonal decomposition of a triangulated category $\mathcal T$ with Serre functor, then $\mathcal A_i$ is admissible for all $i$.  Furthermore, if $\mathcal T =\langle \mathcal A, \mathcal B \rangle$ is a semi-orthogonal decomposition, then $\mathcal B = \leftexp{\perp}{\mathcal A}$.
 \end{lemma}
 
 We now consider the case where $X$ is a smooth projective algebraic variety over $\CC$ and $\mathcal T = \dbcoh{X}$ is the bounded derived category of coherent sheaves on $X$.  As a matter of convention, we denote by the derived pullback, pushforward, (for a morphism $f: X \to Y$) and tensor-product as $f^*,f_*$, and $\otimes$ respectively without further alluding to the fact that they are derived.
 
\begin{definition} Let $X$ and $Y$ be smooth projective algebraic varieties over $\CC$ with $\mathcal P \in \dbcoh{X \times Y}$. Denote the two projections by,
\begin{eqnarray*}
q: X \times Y \to X &\text{and} &p:X \times Y \to Y.
\end{eqnarray*}
\noindent The induced \emph{integral transform} is the functor,
\begin{eqnarray*}
\Phi_{\mathcal P}: \dbcoh{X} \to \dbcoh{Y}& , &\mathcal F \mapsto p_*(q^*\mathcal F \otimes \mathcal P).
\end{eqnarray*}
The object $\mathcal P$ is called the \emph{kernel} of the transform $\Phi_{\mathcal P}$.  Furthermore, the integral transform $\Phi_{\mathcal P}$ is called a \emph{Fourier-Mukai transform} if it is an equivalence.
\end{definition}

\begin{remark}
It is a simple exercise to see that the structure sheaf of the diagonal, $\Delta_*\O_X$, provides the Fourier-Mukai transform corresponding to the identity functor.
\end{remark}

%
%

Integral transforms induce maps between Grothendieck groups and cohomology.
\begin{definition} \label{K integral transform}
Let $X$ and $Y$ be smooth projective algebraic varieties over $k$ and $\mathcal P \in \op{K}_0(X \times Y)$.  The \emph{K-theoretic integral transform} is defined as:
\begin{align*}
\Phi^K_{\mathcal P}: \op{K}_0(X) & \to \op{K}_0(Y)  \\
\mathcal F & \mapsto p_*(\mathcal P \otimes q^*(\mathcal F)).
\end{align*}
Similarly, the  \emph{semi-topological integral transform} is defined as:
\begin{align*}
\Phi^{sst}_{\mathcal P}: \op{K}^{sst}_0(X) & \to \op{K}^{sst}_0(Y) \\
 \mathcal F &  \mapsto p_*(\mathcal P \otimes q^*(\mathcal F)).
\end{align*}
\end{definition}

\begin{definition} \label{cohomological integral transform}
Let $X$ and $Y$ be smooth projective algebraic varieties over $\CC$ and $\alpha \in \emph{H}^*(X \times Y, \Q)$.  The \emph{cohomological integral transform} is defined as:
\begin{align*}
\Phi^H_{\alpha}: \emph{H}^*(X, \Q) & \to \emph{H}^*(Y, \Q) \\
\beta & \mapsto p_*(\alpha \cdot q^*(\beta)).
\end{align*}
For $\mathcal P \in\dbcoh{ X \times Y}$, as shorthand, we set 
\[
\Phi^H_{\mathcal P} := \Phi^H_{\op{ch}(\mathcal P) \cdot \sqrt{\op{td}(X \times Y)}}.
\]
The term, $\op{ch}(\mathcal P) \cdot \sqrt{\op{td}(X \times Y)}$, is called the \emph{Mukai vector} of $\mathcal P$ where $ \sqrt{\op{td}(X \times Y)}$ is a formal square root of the Todd class of $X \times Y$.
\end{definition}

Notice that due to the adjustment by the Mukai vector, the Grothendieck-Hirzebruch-Riemann-Roch formula ensures that an element,  $\mathcal P \in \dbcoh{X \times Y}$, yields a commutative diagram,
\begin{equation} \label{commutative diagram}
\begin{CD}
\dbcoh{X}    @>\Phi_{\mathcal P}>>  \dbcoh{Y}  \\
@VVV                      @VVV                   \\
\op{K}_0(X)    @>\Phi^K_{\mathcal P}>>  \op{K}^{sst}_0(Y)   \\
@VVV                      @VVV                   \\
\op{K}^{sst}_0(X)    @>\Phi^{sst}_{\mathcal P}>>  \op{K}^{sst}_0(Y)   \\
@Vc \circ \op{ch}_{sst}VV                      @Vc \circ \op{ch}_{sst}VV                   \\
\op{H}^*(X, \Q)   @>\Phi^H_{\mathcal P}>>   \op{H}^*(Y, \Q) 
\end{CD}
\end{equation}

 \begin{definition}
Let $X$ and $Y$ be smooth projective algebraic varieties over $k$ and $\mathcal P \in \dbcoh{X \times Y}$.  The \emph{Griffiths integral transform} is defined as:
\begin{align*}
\Phi^{Griff}_{\mathcal P} : {\rm Griff}_{\QQ}(X) & \to {\rm Griff}_{\QQ}(Y) \\
\mathcal F  & \mapsto p_*(\mathcal P \otimes q^*(\mathcal F)),
\end{align*}
i.e., it is the map between kernels $c \circ \op{ch}_{sst}$  induced by the above commutative diagram.
\end{definition}



It was shown by Kuznetsov that the projection functors, $\alpha_i$ are represented by unique integral transforms \cite{Kuz07}.
\begin{proposition} \label{prop: projection kernel}
Let $X$ be a quasiprojective variety over $k$ and 
\[
\dbcoh{X} = \langle \mathcal A_1, ..., \mathcal A_m \rangle
\]
be a semiorthogonal decomposition.  The $i^{\op{th}}$-projection is isomorphic to an integral transform with kernel $P_i$ and $P_i$ is unique up to isomorphism.  Similarly, the $i^{\op{th}}$-truncation is isomorphic to an integral transform with kernel $D_i$ and $D_i$ is unique up to isomorphism.   Furthermore, there is a diagram
\begin{equation} \label{eq: filtration of diagonal}
\begin{tikzpicture}[description/.style={fill=white,inner sep=2pt}]
\matrix (m) [matrix of math nodes, row sep=1em, column sep=1.5em, text height=1.5ex, text depth=0.25ex]
{  0 & & D_{m-1} & \cdots & D_2 & & D_1 & & \Delta_*\O_X   \\
   & & & & & & & &  \\
   & P_m & & & & P_2 & & P_1 & \\ };
\path[->,font=\scriptsize]
 (m-1-1) edge (m-1-3) 
 (m-1-3) edge (m-1-4)
 (m-1-4) edge (m-1-5)
 (m-1-5) edge (m-1-7)
 (m-1-7) edge (m-1-9)

 (m-1-9) edge (m-3-8)
 (m-1-7) edge (m-3-6)
 (m-1-3) edge (m-3-2)

 (m-3-8) edge node[sloped] {$ | $} (m-1-7)
 (m-3-6) edge node[sloped] {$ | $} (m-1-5)
 (m-3-2) edge node[sloped] {$ | $} (m-1-1)
;
\end{tikzpicture}
\end{equation}
 where all triangles are distinguished.
\end{proposition}

\begin{proposition} \label{prop: Griffiths equality}
 Let
 \[
 \dbcoh{X} = \langle \mathcal A_1, ... , \mathcal A_m \rangle
 \]
be a semi-orthogonal decomposition.   There is an isomorphism,
\[
{\rm Griff}_{\QQ}(X) = {\rm Griff}_{\QQ}(\mathcal A_1) \oplus ... \oplus {\rm Griff}_{\QQ}(\mathcal A_m).
\]

\end{proposition}
\begin{proof}
Recall that  $P_i \in \dbcoh{X \times X}$ is the kernel of the $i^{\op{th}}$ projection functor from  Proposition~\ref{prop: projection kernel}.  
Equation~\eqref{eq: filtration of diagonal} yields an equality in $\op{K}_0(X)$ and $\op{K}_0^{sst}(X)$:
\[
\Delta_*\O_X = \sum_{i=1}^m P_i.
\]
Therefore 
\begin{equation*} \label{identity decomposition}
\op{Id} = \Phi^{sst}_{\Delta_*\O_X} = \Phi^{sst}_{\oplus P_i} = \sum_{i=1}^m \Phi^{sst}_{P_i}.
\end{equation*}
Furthermore 
\begin{equation*} \label{idempotents}
\Phi_{P_i} \circ \Phi_{P_j} = 
\begin{cases}
\Phi_{P_i}  & \text{ if } i=j  \\
0 & \text{ if } i \neq j. \\
 \end{cases}
\end{equation*}
By commutativity of \eqref{commutative diagram} this yields 
  \[
\op{Id} = \Phi^{Griff}_{\Delta_*\O_X} = \Phi^{Griff}_{\oplus P_i} = \sum_{i=1}^m \Phi^{Griff}_{P_i}
\]
and
\[
\Phi^{Griff}_{P_i} \circ \Phi^{Griff}_{P_j} = 
\begin{cases}
\Phi^{Griff}_{P_i}  & \text{ if } i=j  \\
0 & \text{ if } i \neq j. \\
 \end{cases}
\]
 The result follows.
%
%

\end{proof}

\begin{lemma} \label{lem: exceptional Griffiths}
Let $E \in \dbcoh{X}$ be exceptional.  Then,
\[
 {\rm Griff}_{\QQ}(\langle E \rangle) = 1
 \]
where $\langle E \rangle$ is the admissible subcategory generated by $E$. 
\end{lemma}
\begin{proof}
We have an equivalence  between $\langle E \rangle$ and the derived category of vector spaces.  Therefore $\op{K}_0(\langle E \rangle) = \ZZ$.  Furthermore, since $E$ is exceptional, the Euler pairing, $\chi(E,E) = 1$.  Therefore by the Grothendieck-Hirzebruch-Riemann-Roch formula,
\[
1 = \chi(E,E) = \langle c(\op{ch}(E)) \cdot \op{td}(X)^{1/2}, c(\op{ch}(E)) \cdot \op{td}(X)^{1/2} \rangle.
\]
Therefore  $\op{ch}(E) \cdot \op{td}(X)^{1/2} \neq 0$ in $\op{H}^{2*}(X, \QQ)$ and since $\op{td}(X)^{1/2}$ is invertible, 
\[
c(\op{ch}(E)) = c_{alg}(\op{ch}_{sst}(E)) \neq 0
\]
 in $\op{H}^{2*}(X, \QQ)$.  Therefore the map,
\[
\QQ = \op{K}_0(\langle E \rangle) \otimes \QQ \to \op{H}^{2*}(X, \QQ)
\]
is injective.

\end{proof}

\begin{corollary} \label{cor: exceptional Griffiths}
Suppose there is a semi-orthogonal decomposition,
\[
\dbcoh{X} = \langle \mathcal A, E_1, ..., E_s \rangle.
\]
There is an isomorphism of Griffiths groups
\[
 {\rm Griff}_{\QQ}(X) =  {\rm Griff}_{\QQ}(\mathcal A).
\]
\end{corollary}
\begin{proof}
This follows immediately from Proposition~\ref{prop: Griffiths equality} and Lemma~\ref{lem: exceptional Griffiths}.
\end{proof}

\medskip

\section{Griffiths Groups for Fano-Calabi-Yaus} \label{sec: CY and FCY}
In this section, we use the general theory from the previous section to relate the Griffiths group of some Fano-Calabi-Yau manifolds to the Griffiths group of an admissible Calabi-Yau category.

Let $A$ be a connected $\Z$-graded commutative 

\begin{definition} \label{defn: Gorenstein}
A $\Z$-graded algebra $A$ is \textbf{Gorenstein} if $A$ has finite $\Z$-graded injective dimension $n$ and there is a integer $a$ such that
\begin{displaymath}
\op{RHom}_A(k, A) \cong k(a)[n].
\end{displaymath}
The element, $\eta(A,M)$, is called the \textbf{Gorenstein parameter} of $(A,M)$. In other words, the derived dual of $k$, $k^{\vee}$, is quasi-isomorphic to $k(\eta(A,M))$. We will simply denote $\eta(A,M)$ by $\eta$ when $(A,M)$ is clear from the context.
\end{definition}

The following definitions are due to Orlov \cite{Orl09}:
\begin{definition}
For a $\Z$-graded algebra, $A$, we defined the graded category of singularities to be the Verdier quotient \cite{Ve}, of the bounded  derived category of finitely generated graded $A$-modules by the category of bounded complexes of graded $A$-modules with finitely generated projective components:
\[
 \op{D}^{\op{gr}}_{\op{sg}}(A) : =  \op{D}^{b}(\op{Mod}-A) / \op{Perf}(A).
\] 
\end{definition}

\begin{definition}
Let $M$ be a finitely generated abelian group.
Let $B = \oplus_{m \in M} B_i$ be a finitely generated $M$-graded commutative algebra over a field $K$.
Consider $w \in B_m$ which is not a zero-divisor.
 The \textbf{category of B-branes} of $w$, is denoted by  $\op{DGrB}(w, M)$. The objects of  $\op{DGrB}(w, M)$ are pairs, 
\begin{center}
\begin{tikzpicture}[description/.style={fill=white,inner sep=2pt}]
\matrix (m) [matrix of math nodes, row sep=3em, column sep=3em, text height=1.5ex, text depth=0.25ex]
{  P_1 & P_0 \\ };
\path[->,font=\scriptsize]
(m-1-1) edge[out=30,in=150] node[above] {$p_1$} (m-1-2)
(m-1-2) edge[out=210, in=330] node[below] {$p_0$} (m-1-1);
\end{tikzpicture}
\end{center}
A morphism $f: P \to Q$ in $\op{DGrB(w, M)}$ is an equivalence class of pairs of morphisms, $f_1: P_1 \to Q_1$ and $f_0 : P_0 \to Q_0$ of degree $0$ such that $f_1(m)p_0 = q_0f_0$ and $q_1f_1=f_0p_1$ where two pairs are equivalent if they are null-homotopic, i.e., if there are two morphisms $s: P_0 \to Q_1$ and $t: P_1 \to Q_0(n)$ such that $f_1 = q_0(m)t+ sp_1$ and $f_0 = t(m)p_0 + q_1s$.
\end{definition}

The objects of  $\op{DGrB}(w, M)$ can also be viewed as quasi-periodic infinite ``complexes'' where the differential squares to $w$ and quasi-periodicity refers to the fact that shifting the complex two to the left is the same as shifting the grading of the modules by $m$. 
From this interpretation, we can define a triangulated structure \cite{Orl09} where $[1]$ is the shift of this complex one to the left.  It follows that we have an isomorphism of functors,
\begin{equation} \label{periodicity equation}
[2] \cong (m),
\end{equation}
in $\op{DGrB}(w, M)$.
Following Orlov, when $M = \Z$ we simply write $\op{DGrB}(w)$.

We state the following special case of Theorem 3.10 in \cite{Orl09}
\begin{theorem} \label{MF = sing}
With the notation above assume that $B$ is regular.  There is an equivalence of categories,
\[
\op{DGrB}(w) \cong  \op{D}^{\op{gr}}_{\op{sg}}(B/w).
\]
\end{theorem}

Let us now recall a special case of a celebrated result of Orlov \cite{Orl09}.
\begin{theorem} \label{thm: Orlov}
Let $X$ be a connected projective Gorenstein scheme of dimension $n$. Let $\mathcal L$ be
a very ample line bundle such that $\omega_X = \mathcal L^{-r}$ for some $r \in \Z$. Set
\[
A := \bigoplus_{i \geq 0} \op{H}^0(X, \mathcal L^i).
\]
\begin{enumerate}
 \item If $r  >0 $, there is a semi-orthogonal decomposition,
\begin{displaymath}
 \dbcoh{X} = \langle \mathcal O(-r),\ldots, \mathcal O(-1), \op{D}^{\op{gr}}_{\op{sg}}(A) \rangle,
\end{displaymath}
 with $\mathcal O(i)$ exceptional.
 \item If $r=0$, there is an equivalence,
\begin{displaymath}
 \dbcoh{X} = \op{D}^{\op{gr}}_{\op{sg}}(A).
\end{displaymath}
 \item If $r < 0 $, there is a semi-orthogonal decomposition,
\begin{displaymath}
\op{D}^{\op{gr}}_{\op{sg}}(A) = \langle k(-r),\ldots, k(-1), \dbcoh{X} \rangle,
\end{displaymath}
where $k(i)$ are exceptional objects corresponding to the $A$-module $\op{H}^0(X, \mathcal L^0)(i)$.
\end{enumerate}
\end{theorem}

\begin{corollary} \label{cor: Griffiths groups relation}
Let $X$ be a smooth projective variety of dimension $n$. Let $\mathcal L$ be
a very ample line bundle such that $\omega_X = \mathcal L^{-r}$ for some $r \geq 0$. Set
\[
A := \bigoplus_{i \geq 0} \op{H}^0(X, \mathcal L^i).
\]
There is an isomorphism of Griffiths groups,
\[
 {\rm Griff}_{\QQ}(X) =  {\rm Griff}_{\QQ}(\op{D}^{\op{gr}}_{\op{sg}}(A) ). 
\]
\end{corollary}

\begin{proof}
From Theorem~\ref{thm: Orlov} there is a semi-orthogonal decomposition,
\[
 \dbcoh{X} = \langle \mathcal O(-r),\ldots, \mathcal O(-1), \op{D}^{\op{gr}}_{\op{sg}}(A) \rangle.
\]
Since $\O(i)$ is exceptional for all $i$ this is an immediate consequence of Corollary~\ref{cor: exceptional Griffiths} 

\end{proof}

\begin{lemma} \label{lem: first two cases}
With the notation above, the category $\op{D}^{\op{gr}}_{\op{sg}}(A)$ is a  Calabi-Yau category of dimension $3$ for the following two cases:

\begin{itemize}

\item  the smooth cubic 7-folds $X_3 \subseteq \PP^8$,

\item the smooth hypersurfaces $X_4 \subseteq \PP^6(1^6;2)$ 
   of degree 4 in the weighted  projective space $\PP^6(1^6;2) = \PP^6(1:1:1:1:1:1:2)$, and 
                  
\end{itemize}

\end{lemma}
\begin{proof}
For a Gorenstein commutative finitely generated connected algebra of injective dimension $n$ and Gorenstein parameter $a$, the Serre functor $S_A$ on 
$\op{D}^{\op{gr}}_{\op{sg}}(A)$ satisfies (see \S 5.3 of \cite{KMV})
\[
S_A = (-a)[n-1].
\]
In the first two cases, $A$ is defined by a single element, $w$, in a $\Z$-graded polynomial ring.  Hence Theorem~\ref{MF = sing} and \eqref{periodicity equation} imply that $\op{deg}(w) \cong [2]$. 

 For $X_3$, $a=6$, $\op{deg}(w) =3$ and $n=8$.  Therefore 
 \[
 S_A = (-6)[7] = [-4][7] = [3].
 \]
 
 For $X_4$, $a=4$,  $\op{deg}(w) =4$ and $n=6$.  Therefore 
 \[
 S_A = (-4)[5] = [-2][5] = [3].
 \]

%


\end{proof}


We are left to consider the final case of a Fano-Calabi-Yau complete intersection in weighted projective space, for which the general member is smooth.  
Let $X_{2.3} \subseteq \PP^7$ be a smooth complete intersections defined by a quadric, $f$, and a cubic, $g$, in $R := k[x_0, ..., x_7]$.  Set $B := R/f$.  We wish to relate the Griffiths group of $X_{2.3}$ to that of a $3$-dimensional Calabi-Yau category.  We know already by Corollary~\ref{cor: exceptional Griffiths} that
\[
{\rm Griff}_{\QQ}(X) =  {\rm Griff}_{\QQ}(\op{D}^{\op{gr}}_{\op{sg}}(R/(f,g)) ).
\]
However, $\op{D}^{\op{gr}}_{\op{sg}}(R/(f,g))$ is not Calabi-Yau in this case.  On the other hand, in light of Theorem~\ref{MF = sing} we expect this category to be closely related to $\op{DGrB}(w)$ (notice that $B$ is not regular so the hypothesis of the theorem is not satisfied).  Indeed, Theorem 3.9 of Orlov states that there is still a fully-faithful functor,
\[
F : \op{DGrB}(w) \to \op{D}^{\op{gr}}_{\op{sg}}(R/(f,g)).
\]  
Moreover, $\op{DGrB}(w)$ is a $3$-dimensional Calabi-Yau category.

\begin{lemma} \label{lem: strange subcategory}                  
Let $X_{2.3} \subseteq \PP^7$ be a smooth complete intersections defined by a quadric, $f$, and a cubic, $g$, in $R := k[x_0, ..., x_7]$.   Consider $g$ as an element of $R/f$.
The category, $\op{DGrB}(g)$, is a $3$-dimensional Calabi-Yau category. 
\end{lemma}

\begin{proof}

To show that $\op{DGrB}(g)$ is a $3$-dimensional Calabi-Yau category, notice that
since $R/f$ is Gorenstein with Gorenstein parameter $6$, the Serre functor on $\op{DGrB}(g)$ is given by $(-6)[7]$.  Since $g$ is a cubic,
\[
(3) = [2] 
\]
by \eqref{periodicity equation} and therefore the Serre functor,
\[
 S_{\op{DGrB}(g)} = (-6)[7] = [-4][7] = [3].
\]
\end{proof}

Unfortunately, $\op{DGrB}(g)$ does not appear to be admissible in $\dbcoh{X_{2.3}}$ and we are unable to apply the discussion in \S\ref{sec: admissible}.  Instead, we can appeal to a closely related category defined by Positselski \cite{Pos1} called the absolute derived category $\op{D}^{\op{abs}}[\mathsf{Fact}(B, \Z, g)]$ (see also \cite{BFK12a, BFK12b}).  
As the details are a bit technical, we just mention that this category a Verdier localization of a category defined the same way as $\op{DGrB}(g)$ except that the pairs consist of any two finitely generated graded $B$-modules as opposed to finitely generated projective graded $B$-modules.  Since $B$ is not regular in our case, this distinction is important.  However, the relevant property that
\begin{equation}
[2] \cong (\op{deg} g) = (3)
\end{equation}
still holds by definition and hence $\op{D}^{\op{abs}}[\mathsf{Fact}(B, \Z, g)]$ is a 3-dimensional Calabi-Yau category.

Furthermore, Positselski establishes an equivalence,
\[
\op{D}^{\op{abs}}[\mathsf{Fact}(B, \Z, g)] \cong \op{D}^{\op{b}}(\op{mod}-B/g)/ \op{D}^{\op{b}}(\op{mod}-B).
\]
Notice that
\[
\op{D}^{\op{gr}}_{\op{sg}}(R/(f,g)) =  \op{D}^{\op{b}}(\op{mod}-R/(f,g)) /  \op{D}^{\op{b}}(\op{mod}-R).
\]
and hence  $\op{D}^{\op{abs}}[\mathsf{Fact}(B, \Z, g)]$ is a Verdier quotient of $\op{D}^{\op{gr}}_{\op{sg}}(R/(f,g))$ by some triangulated subcategory, $\mathcal N \subseteq \op{D}^{\op{gr}}_{\op{sg}}(R/(f,g))$.  From Orlov's theorem, it follows that $\op{D}^{\op{abs}}[\mathsf{Fact}(B, \Z, g)]$ is a Verdier quotient of $\dbcoh{X}$ by the full triangulated subcategory generated by $\mathcal N$ and $\O(-3), \O(-2)$ and $\O(-1)$.

Hence, we arrive at the following result
\begin{proposition} \label{prop: final case}
Let $X_{2.3} \subseteq \PP^7$ be a smooth complete intersections defined by a quadric, $f$, and a cubic, $g$, in $R := k[x_0, ..., x_7]$.  There is a surjective homomorphism of rational vector  spaces:
\[
{\rm Griff}_{\QQ}(X) \to {\rm Griff}_{\QQ}(\op{D}^{\op{abs}}[\mathsf{Fact}(B, \Z, g)]).
\]
Furthermore  $\op{D}^{\op{abs}}[\mathsf{Fact}(B, \Z, g)]$ is a $3$-dimensional Calabi-Yau category.
\end{proposition}
\begin{proof}
This follows immediately from the previous discussion.
\end{proof}

We summarize our results
\begin{theorem}
With the notation above, suppose $X$ is a general smooth Fano-Calabi-Yau complete intersection in weighted projective space.  The Griffiths group,
\[
 {\rm Griff}_{\QQ}(X) =  {\rm Griff}_{\QQ}(\op{D}^{\op{gr}}_{\op{sg}}(A) ),
\]
is infinitely generated.  Furthermore, when $X$ is a cubic $7$-fold or a hypersurface of degree $4$ in $\PP^6(1^6;2)$, then  $\op{D}^{\op{gr}}_{\op{sg}}(A)$ is an admissible $3$-dimensional Calabi-Yau subcategory of $\dbcoh{X}$.
  When $X$ is an intersection of a quadric, $f$, and a cubic, $g$, then ${\rm Griff}_{\QQ}(X)$ surjects onto  $ {\rm Griff}_{\QQ}(\op{D}^{\op{abs}}[\mathsf{Fact}(B, \Z, g)])$ and  $\op{D}^{\op{abs}}[\mathsf{Fact}(B, \Z, g)]$ is a $3$-dimensional Calabi-Yau category.
 \end{theorem}

\begin{proof}
By Corollary~\ref{cor: exceptional Griffiths}, we get the equality,
\[
 {\rm Griff}_{\QQ}(X) =  {\rm Griff}_{\QQ}(\op{D}^{\op{gr}}_{\op{sg}}(A) ).
\]
The fact that it is infinitely generated is a combination of the main result of \cite{AC} and Theorems~\ref{thm: infinitely generated case 1} and \ref{thm: infinitely generated case 2}.

The statement that when $X$ is a cubic $7$-fold or a hypersurface of degree $4$ in $\PP^6(1^6;2)$, then  $\op{D}^{\op{gr}}_{\op{sg}}(A)$ is an admissible $3$-dimensional Calabi-Yau subcategory of $\dbcoh{X}$ is Lemma~\ref{lem: first two cases}.

 When $X$ is an intersection of a quadric, $f$, and a cubic, $g$, the statement is a consequence of Proposition~\ref{prop: final case}
\end{proof}

\begin{remark}
It was pointed out to the authors by M. Kontsevich and B. T\"oen, that using work of T\"oen and Vessosi, one can define the total rational Griffiths group of a saturated dg-category.  This is the kernel of the map from semi-topological K-theory to Hochschild homology tensored over $\Q$.  This definition reduces to our definition in the case of a thick triangulated subcategory of $\dbcoh{X}$.  It would be interesting to see if one can generalize Voisin's result to an open subset of the moduli space of saturated $3$-dimensional Calabi-Yau dg-categories and obtain all FCY manifolds as a special case.
\end{remark}

\medskip

\section{Categorical covers and Griffiths groups} \label{sec: categorical covers}
Suppose $M$ and $N$ are finitely generated abelian groups of rank one and $\phi: M \to N$ is a surjective homomorphism with finite kernel.
When considering the categories, $\op{DGrB}(w, M)$ and $\op{DGrB}(w,N)$, the following abstract situation occurs  (see \cite{BFK11} for a more complete discussion). 


 One has two triangulated categories $ \mathcal T$ and $\mathcal S$, and there is a finite group $\Gamma$ (resp. $\Gamma'$) of autoequivalences of $\mathcal T$ (resp. $\mathcal S$) given for each $\gamma \in \Gamma$ by  $\Upsilon^{\gamma}_{\mathcal T} $ (resp. $\Upsilon^{\gamma'}_{\mathcal S}$).
 Furthermore, there are functors,
\[
F: \mathcal S \to \mathcal T \tand G: \mathcal T \to \mathcal S
\]
satisfying
\[
G \circ F = \bigoplus_{\gamma \in \Gamma} \Upsilon_{\mathcal S} \tand F \circ G = \bigoplus_{\gamma' \in \Gamma'} \Upsilon_{\mathcal T}.
\]

Now suppose that $i_{\mathcal S}: \mathcal S \to \dbcoh{X} $ and $i_{\mathcal T} : \mathcal T \to \dbcoh{Y}$ are admissible subcategories.
Let $P_\mathcal S$ (resp. $P_\mathcal T$) be the projection functor onto $\mathcal S$ (resp $\mathcal T$) with respect to the semi-orthogonal decomposition
\[
\dbcoh{X} = \langle \mathcal S, \text{}^{\perp}\mathcal S \rangle (\text{resp. }\dbcoh{Y} = \langle \mathcal T, \text{}^{\perp}\mathcal T \rangle).
\]
Composing with inclusions and projections, all the functors above can be thought of as functors between some choices of $\dbcoh{X}$ and $\dbcoh{Y}$.  Assume further that all these functors are represented by integral transforms.

In this situation, we obtain an isomorphism between the $\Gamma$-invariant and $\Gamma'$-invariant Griffiths groups,
\begin{equation} \label{covering equation}
 {\rm Griff}_{\QQ}(\mathcal S)^{\Gamma} \to  {\rm Griff}_{\QQ}(\mathcal T)^{\Gamma'},
\end{equation}
with respect to the action of the Griffiths integral transforms induced by the $\Upsilon^{\gamma}_{\mathcal S}$ and $\Upsilon^{\gamma'}_{\mathcal T} $ respectively. 

Now suppose $M$ and $N$ are finitely generated abelian groups of rank one and $\phi: M \to N$ is a surjective homomorphism with finite kernel. Let $B$ be an $M$-graded polynomial algebra.  The map $\phi$ induces an $N$-grading on $B$ as well.
We can let $\mathcal S = \op{DGrB}(w, M)$, $\mathcal T =  \op{DGrB}(w, N)$, $\Gamma= \op{ker }\phi$, and $\Gamma' = \mathbb{G}_{\op{ker }\phi}$ be the dual group.  
For $\gamma' \in \Gamma'$ there is an action on $B$ which acts on a homogeneous element $b \in B_n$ for $n \in N$ by $\gamma'(b) = \gamma'(n) \cdot b$. 
Furthermore we can set $\Upsilon^{\gamma}_{\mathcal T} = (\gamma)$ and  $\Upsilon^{\gamma'}_{\mathcal S} = \op{Id}_\mathcal S$.  One of the central aspects of the work in \cite{BFK11} is that it guarantees that all functors in question are always represented by integral transforms.  
\begin{proposition}
Let $M$ and $N$ be finitely generated abelian groups of rank one and $\phi: M \to N$ be a surjective homomorphism with finite kernel and $B$ be an $M$-graded polynomial algebra.  Suppose that $\op{DGrB}(w, M), \op{DGrB}(w, N)$ are admissible subcategories of $\dbcoh{X}$ and $\dbcoh{Y}$ respectively where $X$ and $Y$ are smooth proper algebraic varieties over $k$. There is an isomorphism,
\[
 {\rm Griff}_{\QQ}(\op{DGrB}(w, M))^{\op{ker }\phi} \to  {\rm Griff}_{\QQ}(\op{DGrB}(w, N))^{\op{ker }\phi}.
\]
\end{proposition}

\begin{example} \label{3 curves}
Let $M=(\Z \oplus \Z \oplus \Z)/\langle(3,-3,0), (3,0,-3) \rangle$ and $B = k[x_0, ..., x_8]$ be the $M$-graded algebra where,
\[
\op{deg}(x_i)  = 
\begin{cases}
 (1,0,0) & \text{ if } 0 \leq i \leq  2,  \\
(0,1,0) & \text{ if } 3 \leq i \leq 5, \\
(0,0,1) & \text{ if } 6 \leq i \leq 8. \\
 \end{cases}
\]
  Consider smooth elliptic curves, $E_1,E_2,E_3$ defined by $f_3(x_0,x_1,x_2), g_3(x_3,x_4,x_5),$ and $h_3(x_6,x_7,x_8)$.  Set $w=f_3+g_3+h_3$.
By \cite{BFK11} there is an equivalence of categories,
\[
\dbcoh{E_1 \times E_2 \times E_3} \cong \op{DGrB}(w, M).
\]
Let $N = \Z$ and $\phi$ be the summing map.  The kernel of $\phi$ is $\Z_3 \oplus \Z_3$, a finite group of order $9$ generated by $(1,-1,0)$ and $(1,0,-1)$.  By Theorem~\ref{thm: Orlov} $\op{DGrB}(w, N)$ is an admissible subcategory of $\dbcoh{X_3}$ where $X_3$ is the cubic sevenfold defined by $w$.  Hence, by Corollary~\ref{cor: Griffiths groups relation} we obtain:
\[
 {\rm Griff}_{\QQ}(E_1 \times E_2 \times E_3)^{\Z_3 \oplus \Z_3} \cong  {\rm Griff}_{\QQ}(X_3)^{\Z_3 \oplus \Z_3}.
\]
\end{example}


\begin{example} \label{fourfold by curve}
Let $M=(\Z \oplus \Z )/\langle(3,-3) \rangle$ and $B = k[x_0, ..., x_8]$ be the $M$-graded algebra where,
\[
\op{deg}(x_i)  = 
\begin{cases}
 (1,0) & \text{ if } 0 \leq i \leq  5,  \\
(0,1) & \text{ if } 6 \leq i \leq 8. \\
 \end{cases}
\]
Let $Z$ be a smooth cubic-fourfold by $f_3(x_0, ..., x_5)$ and $E$ be a smooth elliptic curve defined by $h_3(x_6,x_7,x_8)$.  Set $w=f+h$.
Suppose further that $\op{DGrR}(f) \cong \dbcoh{Y}$ for some smooth $K3$ surface, $Y$, where $R=k[x_0, ..., x_5]$.\footnote{This occurs, for example, when $f$ is Pffafian or contains a plane $P$ and a 2-dimensional cycle $T$ such that $T \cdot H^2 - T \cdot P$ is odd \cite{Kuz09b}.}
By \cite{BFK11} there is an equivalence of categories,
\[
\dbcoh{Y \times E} \cong \op{DGrB}(w, M).
\]
Let $N = \Z$ and $\phi$ be the summing map.  The kernel of $\phi$ is $\Z_3$, the finite group of order $3$ generated by $(1,-1)$.  By Theorem~\ref{thm: Orlov} $\op{DGrB}(w, N)$ is an admissible subcategory of $\dbcoh{X_3}$ where $X_3$ is the cubic sevenfold defined by $w$.  Hence, by Corollary~\ref{cor: Griffiths groups relation} we obtain:
\[
 {\rm Griff}_{\QQ}(Y \times E)^{\Z_3} \cong  {\rm Griff}_{\QQ}(X_3)^{\Z_3}.
\]
\end{example}

\begin{example} \label{K3 by curve}
Let $M=(\Z \oplus \Z )/\langle(2,-4) \rangle$ and $B = k[x_0, ..., x_6]$ be the $M$-graded algebra where,
\[
\op{deg}(x_i)  = 
\begin{cases}
 (1,0) & \text{ if } 0 \leq i \leq  3,  \\
(0,1) & \text{ if } 4 \leq i \leq 5. \\
(0,2) & \text{ if } i=6. \\
 \end{cases}
\]
Let $Y$ be a smooth quartic $K3$-surface defined by $f_4(x_0, ..., x_3)$ and $E$ be a smooth elliptic curve defined by $h_4(x_4,x_5,x_6)$ in $\P(1:1:2)$.  Set $w=f_4+h_4$.
By \cite{BFK11} there is an equivalence of categories,
\[
\dbcoh{Y \times E} \cong \op{DGrB}(w, M).
\]
Let $N = \Z$ and $\phi$ be the summing map.  The kernel of $\phi$ is $\Z_2$, the finite group of order $2$ generated by $(1,-2)$.  By Theorem~\ref{thm: Orlov} $\op{DGrB}(w, N)$ is an admissible subcategory of $\dbcoh{X_4}$ where $X_4$ is the smooth hypersurface in $\P(1^6;2)$ defined by $w$.  Hence, by Corollary~\ref{cor: Griffiths groups relation} we obtain:
\[
 {\rm Griff}_{\QQ}(Y \times E)^{\Z_2} \cong  {\rm Griff}_{\QQ}(X_4)^{\Z_2}.
\]
\end{example}

In light of these examples, let us propose the following conjectures:
\begin{conjecture}
For the general member of each of the families above, the invariant Griffiths groups,
\[
 {\rm Griff}_{\QQ}(X_3)^{\Z_3 \oplus \Z_3} \tand  {\rm Griff}_{\QQ}(X_4)^{\Z_2} 
\]
are infinitely generated.
\end{conjecture}

From the examples above, it follows that the product of three general elliptic curves, $E_1 \times E_2 \times E_3$, is infinitely generated as well as the product of a general quartic $K3$-surface and a general elliptic curve in $\P(1:1:2)$.

\bigskip


\bigskip


\end{document}